\documentclass[a4paper,12pt]{article}
 
\usepackage{amsmath,amsthm,amssymb}
\usepackage{graphicx}
\usepackage{dsfont}
\usepackage{enumerate}
\usepackage{mathrsfs}
\usepackage[colorlinks]{hyperref}

\newtheorem{theorem}{Theorem}[section]
\newtheorem{lemma}[theorem]{Lemma}

\theoremstyle{definition}

\theoremstyle{remark}

\numberwithin{equation}{section}
\newcommand{\E}{\mathds{E}}
\newcommand{\Q}{\mathds{Q}}
\renewcommand{\P}{\mathds{P}}

\newcommand{\R}{\mathds{R}}
\newcommand{\N}{\mathds{N}}
\renewcommand{\S}{\mathds{S}}
\newcommand{\B}{B^d}

\newcommand{\PP}{\mathcal{P}}
\newcommand{\PPc}{\PP_\cent}
\newcommand{\PPn}[1][n]{\PP_{#1}}
\newcommand{\PPnc}[1][n]{\PP_{#1,\cent}}

\newcommand{\PPcs}[1][n]{\PP_{\cent,\Phi}}

\newcommand{\PPncs}[1][n]{\PP_{#1,\cent,\Phi}}
\newcommand{\KK}{\mathcal{K}}

\newcommand{\KKcs}{\KK_{\cent,\Phi}}
\newcommand{\HH}{\mathcal{H}}
\newcommand{\HS}{\widetilde{\HH}}

\newcommand{\MeasPPn}[1][n]{\Theta_{#1}}
\newcommand{\MeasPPnMu}[1][n]{\mu_{#1}}

\newcommand{\MeasPPncsMu}[1][n]{\mu_{#1,\cent,\Phi}}

\newcommand{\MeasLebesgue}{\lambda_d}
\newcommand{\MeasLebesgueModif}[1]{\lambda_1^{(#1)}}

\newcommand{\MeasMultiSimple}[3]{\dint  #1^{#2} \left( \bd{#3} \right) }
\newcommand{\MeasCount}{\delta}
\newcommand{\MeasHalfSpaceMu}{\tilde{\mu}}

\newcommand{\IntMos}{\gamma^{(d)}}

\newcommand{\bd}{\boldsymbol}
\newcommand{\1}{\mathds{1}}
\newcommand{\e}{\varepsilon} 
\newcommand{\dint}{\,\mathrm{d}}
\newcommand{\cent}{\mathfrak{c}}
\newcommand{\origin}{\boldsymbol{o}}
\newcommand{\shape}{\mathfrak{s}}

\newcommand{\vect}[1]{\boldsymbol{ #1 }}

\usepackage{constants}
\newconstantfamily{univ}{
  symbol={c}, 
}

\begin{document}

\title{Cells with many facets in a \\ Poisson hyperplane tessellation}

\author{
  Gilles Bonnet\thanks{
    Institut f\"{u}r Mathematik, Universit\"{a}t Osnabr\"{u}ck, Albrechtstr. 28a, 49076 Osnabr\"{u}ck, Germany.
    Email: gilles.bonnet@uni-osnabrueck.de} \and 
  Pierre Calka\thanks{
    Laboratoire de Math\'{e}matiques Rapha\"{e}l Salem, Universit\'{e} de Rouen, Avenue de l'Universit\'{e}, BP.12, Technop\^{o}le du Madrillet, F76801 Saint-Etienne-duRouvray France.
    Email: pierre.calka@univ-rouen.fr} \and 
  Matthias Reitzner\thanks{
    Institut f\"{u}r Mathematik, Universit\"{a}t Osnabr\"{u}ck, Albrechtstr. 28a, 49076 Osnabr\"{u}ck, Germany.
    Email: matthias.reitzner@uni-osnabrueck.de}}

\date{}

\maketitle

\begin{abstract}
  Let $Z$ be the typical cell of a stationary Poisson hyperplane tessellation in $\R^d$. The distribution of the number of facets $f(Z)$ of the typical cell is investigated. It is shown, that under a {\it well-spread} condition on the directional distribution, the quantity $n^{\frac{2}{d-1}}\sqrt[n]{\P(f(Z)=n)}$ is bounded from above and from below. When $f(Z)$ is large, the isoperimetric ratio of $Z$ is bounded away from zero with high probability.
  
  These results rely on one hand on the Complementary Theorem which provides a precise decomposition of the distribution of $Z$ and on the other hand on several geometric estimates related to the approximation of polytopes by polytopes with fewer facets.
  
  From the asymptotics of the distribution of $f(Z)$, tail estimates for the so-called $\Phi$ content of $Z$ are derived as well as results on the conditional distribution of $Z$ when its $\Phi$ content is large.
\end{abstract}

\section{Introduction}
One of the classical models in stochastic geometry to generate a random mosaic is the construction via a Poisson hyperplane process.
A Poisson hyperplane process consists of countably many random hyperplanes in $\R^d$ chosen in such a way, that their distribution is translation invariant, the distribution of the direction of the hyperplanes follows a directional distribution $\varphi$, and the number of hyperplanes hitting an arbitrary convex set $K$ is Poisson distributed.

Such a Poisson hyperplane process tessellates $\R^d$ into countably many convex polytopes, the tiles of the mosaic.
The distribution of a tile chosen at random is the distribution of the so-called {\it typical cell} $Z$, a random polytope. 

The typical cell has been investigated intensively in the past decades, numerous papers have been dedicated to describe quantities associated with this cell, for example volume, surface area, mean width, number of facets, etc.
The expected number of facets $f(Z)$ of the typical cell and the expected volume $V_d(Z)$ are known, see e.g. the first works due to Miles \cite{Miles72,Miles73} and Matheron \cite{Matheron} as well as Chapter 10 from the seminal book of Schneider and Weil \cite{SchnWe3} and the survey \cite{Calka10}. 

But in almost all cases the distribution of these quantities is out of reach, and even good approximations are extremely difficult and unknown so far.
Our main theorem fills this gap for the number of facets of $Z$, giving precise asymptotics for the tails of the distribution.
\begin{theorem}
  \label{mainthm:UpperBoundfn}
  There exists a constant $ \Cl[univ]{upperbound2TH1}>0$, depending on $\varphi$,  such that for ${n\ge d+1}$,
  \[
    \P ( f(Z)=n)     <  \Cr{upperbound2TH1}^n \, n^{-\frac{2n}{d-1}}.
  \]
  Furthermore, there exists an integer $n_\varphi$ such that $\P(f(Z)=n)$ is either vanishing or strictly decreasing for $n \geq n_\varphi$.
\end{theorem}
Here and in the sequel, $c_i$ will denote a positive constant which depends on dimension $d$.
It will be specified when it depends on $\varphi$ or another parameter.

It is clear that in general there is no matching lower bound, for example if the directions of the hyperplane process are concentrated on a finite set.
We prove that, if the directional distribution satisfies a mild condition, we have lower bounds of the same order in $n$ as the upper bound above.
In the following, we call $\varphi$ {\it well spread} if there exists a cap on the unit sphere where $\varphi$ is bounded from below by a multiple of the surface area measure.
\begin{theorem}
  \label{mainthm:LowerBoundfn}
  Assume that $\varphi$ is well spread.
  Then there exists a constant $\Cl[univ]{lowerbound2TH2} > 0$, depending on $\varphi$, such that for $n\ge d+1$,
  \[
    \P ( f(Z)=n)   >  \Cr{lowerbound2TH2}^n \,  n^{- \frac{2n}{d-1}}.
  \]
\end{theorem}
The occurring constant will be made more explicit in Section \ref{sec:manyfac}, in particular its dependence on the directional distribution $\varphi$ of  $\eta$. 

Maybe a simple conjecture for the distribution of the number of facets $f(Z)$ of the typical cell of a Poisson hyperplane tessellation would have been the Poisson distribution.
Yet our theorem disproves this, as a more intricate conjecture we state that $f(Z)$ follows a Compound Poisson distribution.

Theorems \ref{mainthm:UpperBoundfn} and \ref{mainthm:LowerBoundfn} prove the asymptotic expansion
\[
  \ln  \P ( f(Z)=n)   =  - \frac{2}{d-1} \, n \ln n +  \Theta(n) 
\]
as $n \to \infty$ where the implicit constants in the error term $\Theta(n)$ are strictly positive.
We pose it as an open problem whether
\[
  \ln  \P ( f(Z)=n)   =  - \frac{2}{d-1} \, n \ln n +  c n + \Theta(\ln n).
\]
Support for this conjecture comes from the planar case where Calka and Hilhorst \cite{CaHi08} stated a more precise result if $\varphi$ is rotation invariant.  
In \cite{H09}, Hilhorst investigates the similar case of the typical cell of a Poisson-Voronoi tessellation and provides heuristics for getting an analogous asymptotic expansion for the probability that the typical Poisson-Voronoi cell has $n$ facets.
He obtains as a first term $-\frac{2}{d-1}n\ln(n)$, as in the present paper, but does not make the second term fully explicit (indeed, the constant $c_d$ introduced in (2.10) therein is unknown).
In particular, the approximation result provided by Lemma \ref{lem:RSW3} matches to some extent the statement (2.10) in \cite{H09}, i.e. {\it many} of the $n$ facets of $Z$ lie in an annulus with thickness of order $n^{-\frac{2}{d-1}}$ multiplied by a size functional of $Z$.
This suggests that improving Theorems \ref{mainthm:UpperBoundfn} and \ref{mainthm:LowerBoundfn} requires new ingredients and notably a substantial improvement of Lemma \ref{lem:RSW3}.

\medskip
In the following our aim is to show that cells with many facets are far away from any lower dimensional convex body. To do this we measure the distance from the ball using the isoperimetric ratio of a convex set $K$. Denote by $V_i$ the $i$-th intrinsic volume (see Section \ref{sec:Setting} for the definition). 
For any $ 1 \leq i < j \leq d $ we call $ V_j(K)^{1/j} V_i(K)^{-1/i} $ the {$(i,j)$-isoperimetric ratio} of $K$.
The isoperimetric inequality says that this ratio is maximized precisely for balls.
On the other hand when the isoperimetric ratio of $K$ vanishes, $K$  must be lower dimensional.

The next theorem shows that the isoperimetric ratio $ V_j(Z)^{1/j} V_i(Z)^{-1/i} $ of the typical cell is bounded away from zero with high probability  if the cell has many facets.
Hence cells with many facets cannot be too elongated.

\begin{theorem}
  \label{mainthm:boundsManyFacetsIsoperimeter}
  Assume that $\varphi$ is well spread and that $1 \leq i < j \leq \lceil (d-1) / 2 \rceil $.
  For any $\delta\in(0,1)$, there exist constants $\epsilon$ and $n_0$, depending on $\varphi$, $i$, $j$ and $\delta$ such that
  \[
    \P \left( \frac{V_j(Z)^{\frac 1j}}{V_i(Z)^{\frac 1i}} < \epsilon \ \Big\vert f(Z)=n\right) 
    \leq  
    \delta^n  
  \]
  for all $n \geq n_0$.
\end{theorem}

\bigskip
To describe the distribution of the typical cell $Z$ we need the notion of the \textit{$\Phi$-content} of a convex body $K$. 
The $\Phi$-content measures in a certain sense the size of the convex set depending on the directional distribution $\varphi$ of the hyperplane tessellation.
It is given by
\[
  \Phi(K)
  := \int\limits_{\S^{d-1}} h(K, \vect{u} )\dint \varphi(\vect{u} ) ,
\]
where $h(K, \vect{u} ):=\max\{\langle \vect{x}, \vect{u} \rangle\mid \vect{x}\in K\}$ is the value of the {support function} of $K$ at $ \vect{u} $.
That the $\Phi$ content is an important quantity of a  Poisson hyperplane process  is immediately clear since the number of hyperplanes hitting an arbitrary convex set $K$ is Poisson distributed with parameter $\gamma \Phi(K)$.
The real number $\gamma>0$ is the intensity of the hyperplane process.
In the important case where $\varphi$ is a constant and hence the directional distribution is the uniform distribution on $\S^{d-1}$, the $\Phi$-content of a convex set $K$ is just the well known mean width $V_1(K)$ of $K$ up to a constant.
For more information we refer to Section \ref{sec:Setting}.

Again the distribution of $\Phi(Z)$ is unknown, and in this case even the expectation is out of reach.
Here we succeed in computing the tail behaviour of the size-functional $\Phi (Z)$.
\begin{theorem}
  \label{mainthm:DistribPhi}
  There exist constants $\Cl[univ]{const:UpperBoundPhiContent} > \Cl[univ]{const:LowerBoundPhiContent} > 0$ and  $\Cl[univ]{const:LowerA}>0$ depending on $\varphi$, such that the following holds.
  For $a>0$, we have
  \[
    \P ( \Phi(Z) > a ) 
    < \exp \left\{ - \gamma a  + \Cr{const:UpperBoundPhiContent} ( \gamma a ) ^{\frac{d-1}{d+1}} \right\} . 
  \]
  Assume that $\varphi$ is well spread.
  Then, for $ a > \gamma^{-1} \Cr{const:LowerA}$, we also have
  \[
    \P ( \Phi(Z) > a )
    > \exp \left\{ - \gamma a 
    + \Cr{const:LowerBoundPhiContent} ( \gamma a )^{\frac{d-1}{d+1}} \right\} .
  \]
\end{theorem}

Again we can use our bounds on the distribution of the $\Phi$-content of the typical cell to show that big cells are not too elongated.
\begin{theorem}
  \label{mainthm:FlatBigPhi}
  Assume that $\varphi$ is well spread and that $1 \leq i < j \leq \lceil (d-1)/2 \rceil$.
  Then for any  $ \epsilon > 0 $ sufficiently small  we have 
  \[
    \lim_{a \to \infty }
    \P \left( \frac{V_j(Z)^{\frac 1j}}{V_i(Z)^{\frac 1i}} < \epsilon \mid \Phi(Z)>a \right)
    =0 .
  \]
\end{theorem}

It is a long standing general conjecture that extremal cells of Poisson hyperplane mosaics converge to a limit shape.
This question is known as  Kendall's problem and has attained great interest with a large number of important contributions, see \cite{Kov97} \cite{HSR1} \cite{HS7} \cite{HS4} \cite{HS5} and the surveys \cite{Calka10} \cite{Hugsurv-rdmosaic} \cite{HugReisurv} \cite{Calka13}.
For a precise definition of the {\it shape} $\shape (Z)$, of the cell $Z$ we refer to Section \ref{sec:Setting}.
It turned out that many size functionals allow positive solutions of Kendall's problem, but the first intrinsic volume, the $\Phi$-content and the number of facets resisted all attemps so far.
In the contrary, Hug and Schneider \cite[Thm. 4]{HS7} gave an example where the shape of the cell containing the origin, i.e. the zero cell, under the condition that is has a big $\Phi$-content does \textit{not} concentrate.

Theorems \ref{mainthm:boundsManyFacetsIsoperimeter} and \ref{mainthm:FlatBigPhi} are first attemps to close the existing gaps.
They show that the shape $\shape (Z)$ of the typical cell cannot be too elongated if either the number of facets or the $\Phi$-content is large.
At a first glance they seem to be in conflict with the example given by Hug and Schneider.
Yet for their example they used a measure $\varphi$ which is concentrated on finitely many points and thus not well spread.
Next we prove that their theorem holds when the zero cell is replaced by the typical cell.
\begin{theorem}\label{mainthm:nmaxshape}
  Assume $\varphi$ is concentrated on a finite number of points, $f(Z) \leq n_{\max}$ with probability one.
  Then there is a limiting shape distribution,
  \[
    \lim_{a \to \infty}
    \P(\shape (Z) \in S | \Phi(Z)\geq a) =
    \P(\shape(Z) \in S | f(Z)=n_{\max}) .
  \]
\end{theorem}
Note that when conditioning on the number of facets $f(Z)$, the shape $\shape(Z)$ of the typical cell is independent of the size $\Phi(Z)$ (see the Complementary Theorem \ref{thm:ComplementaryThm}) and is given explicitly in Theorem \ref{thm:ComplementaryThm}.

The paper is organized as follows.
In Section \ref{sec:notandcompl}, we fix the general setting and introduce the so-called Complementary Theorem, which is a practical disintegration of the distribution of the typical cell $Z$.
This is the fundamental probabilistic tool for showing our main results.
In Section \ref{sec:PolytopeApproximation}, we provide the required geometric ingredients which deal with the approximation of polytopes by polytopes with fewer facets.
Section \ref{sec:manyfac} is devoted to the proof of the main results about cells with many facets, i.e. Theorems \ref{mainthm:UpperBoundfn}, \ref{mainthm:LowerBoundfn} and \ref{mainthm:boundsManyFacetsIsoperimeter}.
Finally, in Section \ref{sec:bigcells}, we prove Theorems \ref{mainthm:DistribPhi}, \ref{mainthm:FlatBigPhi} and \ref{mainthm:nmaxshape} which deal with the {\it big cells}, i.e. with a large $\Phi$ content. 
\section{Notations and the Complementary Theorem}\label{sec:notandcompl}
\subsection{Setting and notations}\label{sec:Setting}
As standard references to the following material from convex geometry we refer to the books by Schneider \cite{Schn42nd} and Gruber \cite{Gruberbook}

We work in a $d$-dimensional Euclidean vector space $\R^d$, $d\geq2$, with scalar product $\langle\cdot,\cdot\rangle$, norm $\lVert\cdot\rVert$ and origin $ \origin $.
We denote by $ B ( \vect{x} , r ) $ the closed ball and by $ S ( \vect{x} , r ) = \partial B ( \vect{x} , r ) $ the sphere with center $ \vect{x} $ and radius $ r $, by  $ \B = B ( \origin , 1 ) $ the unit ball and  by $ \S^{d-1} = \partial \B $ the unit sphere.
Let $\HH $ be the space of affine hyperplanes in $\R^d$ with its usual topology and Borel structure.
Every hyperplane $H\in\HH $ has a unique representation 
\[
  H( \vect{u} ,t)
  :=\{ \vect{x}\in\R^d :\  \langle \vect{x}, \vect{u} \rangle=t\}.
\]
with $ \vect{u} \in \S^{d-1}$ and $t>0$. 
For a given hyperplane $H\in\HH $, we write $H^-$, resp. $H^+$, for the closed halfspace with boundary $H$ which contains, resp. excludes the origin, 
\[
  H( \vect{u} ,t)^- = \{ \vect{x}\in\R^d :\ \langle \vect{x}, \vect{u} \rangle \leq t\}
  \quad\text{ and }\quad 
  H( \vect{u} ,t)^+ = \{ \vect{x}\in\R^d :\  \langle \vect{x}, \vect{u} \rangle \geq t\}.
\]
We denote by $ \HS := \HH  \times \{ \pm 1 \} $ the space of halfspaces.

Let $\KK $ be the set of convex bodies (compact convex sets of $\R^d$ with non-empty interior) and denote by $K^o$ the relative interior of a set $K \in \KK$.
We write $\PP $ for the set of all polytopes, and $f(P)$ for the number of facets of a polytope $P \in \PP$.
Denote by $\PPn =\{P \in \PP| f(P)=n \} $ the set of $n$-topes, hence $\PPn \subset \PP \subset\KK$.
For any $ t > 0 $, and $ K , L \in \KK $ we define
\[
  t K := \{ t \vect{x} :\  \vect{x} \in K \},
  \quad
  K + L := \{ \vect{x} + \vect{y} :\  \vect{x} \in K , \vect{y} \in L \}
\]
where the latter is the Minkowski sum of $K$ and $L$.

The sets $\KK$, $\PP$, and $\PPn$ are equipped with the Hausdorff distance $d_H$, 
\[
  d_H(K,L) = \min \{r\colon  K\subset L+rB^d,\ L \subset K+rB^d \} ,
\]
and with the associated topology and Borel structure.

Steiner's formula says that the volume of the Minkowski sum of $K \in \KK$ and a ball of radius $r$ is a polynomial in $r$,
\[
  V_d(K + r \B)= \sum_{i=0}^d \kappa_i V_{d-i}(K) r^{i}  
\]
where $\kappa_i$ denotes the volume of the $i$-dimensional unit ball, and $V_i(K)$ is the $i$-th \textit{intrinsic volume} of $K$.
E.g., $V_d$ is the usual volume, $2 V_{d-1}$ the surface area and $V_1$ a multiple of the mean width of $K$. 
Steiner's formula can be generalized to all intrinsic volumes, 
\begin{equation}
  \label{eq:steinerintrinsic}
  V_j(K + r \B)= \sum_{i=0}^j \binom{d-j+i}{i} \frac{\kappa_{d-j+i}}{\kappa_{d-j}}  V_{j-i}(K) r^{i}  .
\end{equation}
The isoperimetric inequality tells us that for $1 \leq i < j \leq d$ the isoperimetric ratio is bounded,
\begin{equation}
  \label{eq:isopinequ}
  \frac {V_j(K)^{\frac 1j}}{V_i(K)^{\frac 1i}} \leq \frac {\kappa_j^{\frac 1j}} {\kappa_i^{\frac 1i}}  
\end{equation}
with equality if and only if $K$ is a ball.
The ratio is well defined if the dimension of $K$ is at least $i$ and equals zero if and only if the dimension is at most $(j-1)$.
In Theorems \ref{mainthm:boundsManyFacetsIsoperimeter} and \ref{mainthm:FlatBigPhi} we prove that with high probability the isoperimetric ratio of the cells with many facets is bounded away from zero, thus proving that it is not too close to a $(j-1)$-dimensional convex body.

Let $\eta$ be a stationary Poisson hyperplane process in $\R^d$, that is a Poisson point process in the space $\HH $ invariant under translation.
We often identify a simple counting measure with its support, so that for any set $ \mathcal{A} \subset \HH  $, both notations $\eta(\mathcal{A})$ and $| \eta \cap \mathcal{A} |$ denote the number of elements of $\eta$ in $\mathcal{A}$.

Since $\eta$ is stationary, its {intensity measure} $\E \, \eta ( \cdot ) $ decomposes into Lebesgue measure and an even probability measure $ \varphi $ on $\S^{d-1}$,
\begin{equation}\label{def:Thetamu}
  \E \, \eta ( \cdot )
  = \gamma  \mu (\cdot)
  := \gamma \int\limits_{\S^{d-1}}\int\limits_0^\infty \1 \left( H( \vect{u} ,t)\in\cdot \right) \dint t \dint \varphi( \vect{u} ) , 
\end{equation}
where $ \gamma >0$ and $\mu$ is a measure  on $\HH $.
We call $ \gamma $ the {intensity} and $\varphi$ the {directional distribution} of the hyperplane process $\eta$.
We assume that the support of $\varphi$ is not contained in a great circle of $\S^{d-1}$.
When $\varphi$ is the normalized surface area measure on $\S^{d-1}$, we say that $\eta$ is {isotropic}.

The closure of each of the connected components of the complement of the union $\bigcup_{H\in\eta} H$ is almost surely a polytope (because the support of $\varphi$ is not contained in a great circle).
These polytopes are the {cells} of the Poisson hyperplane mosaic $X$ induced by $\eta$.
We can see $X$ as a point process in $\PP$.
To describe the distribution of $X$ we need the notion of the \textit{$\Phi$-content} of a convex body $K$ which is given by
\[
  \Phi(K)
  := \int\limits_{\S^{d-1}} h(K, \vect{u} )\dint \varphi(\vect{u} ) ,
\]
where $h(K, \vect{u} ):=\max\{\langle \vect{x}, \vect{u} \rangle:\  \vect{x}\in K\}$ is the value of the {support function} of $K$ at $ \vect{u} $.
In the important case when $\varphi$ is a constant and hence the directional distribution is the uniform distribution on $\S^{d-1}$, the $\Phi$-content of a convex set $K$ is up to a constant just the first intrinsic volume $V_1(K)$.

We will have to replace bounds expressed in terms of the first intrinsic volume $V_1 (K) $ by bounds expressed in terms of $\Phi$. 
Because $V_1$ and $\Phi$ are homogeneous of degree one, and $\Phi$ is continuous and strictly positive on the compact set $\{ \mbox{convex compact sets} \ K\subset \R^d :\  V_1(K)=1 \mbox{ and } \origin\in K\}$, we see that $0 < c_\Phi := \sup_{K \in \KK} V_1(K) / \Phi(K)   < \infty $.
Thus for all $K \in \KK$,
\begin{equation}
  \label{eq:V1RPhi}
  V_1(K) \leq c_\Phi \Phi ( K ) . 
\end{equation}
The definition of $\Phi$ is motivated by the fact that, for any $K\in\KK $,
\[
  \P (\eta\cap K=\emptyset)
  = e^{- \gamma \Phi(K)} .
\]
Note that $\Phi$ is homogeneous of degree $1$ and translation invariant, $\Phi( t K + \vect{x} )  = t \Phi(K)$ for any $ K \in \KK $, $t \geq 0 $ and $ \vect{x} \in \R^d $.
In the special case where $\eta$ is isotropic, $ 2 \Phi $ is the so called {mean width} of $K$.
Note that $K \subset \KK$ implies that  $\Phi(K)>0$ since $K$ contains at least $2$ points.
For a set $\mathcal{X}\subset\KK$ of convex bodies we define
\[
  \mathcal{X}_\Phi
  := \{ K\in\mathcal{X} :\  \Phi (K) = 1 \} \subset {\mathcal X}.
\]

Let $\cent:\KK\to\R^d$ be a \textit{center function}, i.e. a measurable map compatible with translations and homogeneous under the scale action: 
\[
  \cent( t K+ \vect{x})=t\cent(K)+ \vect{x}
\]
for any $ t \in (0,\infty) $ and any $ \vect{x} \in \R^d $.
For example, $\cent$ can be the center of mass.
In this paper we assume that $\cent(K)\in K$ for every $K\in\KK$ and that  $\cent$ is $1$-Lipschitz, i.e. $\|\cent(K)-\cent(L)\| \leq d_H(K,L)$ for $K,L\in\KK$.
For a set $\mathcal{X}\subset\KK$ of convex bodies we define
\[
  \mathcal{X}_{\cent}
  := \{ K\in\mathcal{X} :\  \cent (K)= \origin \} \subset {\mathcal X}.
\]
In particular, $\PP_\cent$ denotes the set of polytopes with center at the origin.
Due to the natural homeomorphism
\begin{eqnarray*}
  \PP & \to & \R^d \times \PPc\\
  P & \mapsto & \left( \cent (P), P - \cent(P) \right),
\end{eqnarray*}
we will consider from now on $X$ as a {germ-grain process} in $\R^d$ with {grain space} $\PPc$.
Since $\eta$ is stationary, this is also the case for $X$.
That implies the existence of a probability measure $\Q$ on $\PPc$ such that the intensity measure of the germ-grain process $X$ decomposes into 
$\Q$ and Lebesgue measure $\MeasLebesgue$,
\begin{equation}
  \label{eq:intSplit}
  \E\, X( \{ P-\cent(P) \in C ,\,  \cent (P) \in A \}) =
  \IntMos  \MeasLebesgue(A)\,  \Q(C) 
\end{equation}
for $C \subset \PPc$.
We call $\Q$ the {grain distribution}, and the constant 
$ \IntMos   = \E\, X( \{ P \in \PP ,\, \cent (P) \in [0,1]^d \}) $ 
the {intensity of $X$}.
It is easy to see that $\IntMos$ is a multiple of $\gamma^d$, where $\gamma$ is the intensity of the Poisson hyperplane process.
A random centred polytope $Z\in\PPc$ with distribution $\Q$ is called {typical cell} of~$X$.

For $ K \in \KK $, we define its \textit{shape} to be
\[
  \shape (K) = \frac{1}{\Phi(K)} ( K - \cent (K) ) . 
\]
We have that $\shape$ is translation and scale invariant, i.e. for any $K\in\KK$, $t\in(0,\infty)$ and any $ \vect{x} \in \R^d $, we have
\[
  \shape ( t K + \vect{x} )
  = \shape ( K ).
\]
We want to point out that in this paper the shape is not rotation invariant.
We call the set $\KKcs = \shape (\KK)$ the \textit{shape space}.
Similarly we call $\PPcs = \shape (\PP)$ the \textit{shape space of polytopes}, and $\PPncs = \shape (\PPn)$ the \textit{shape space of $n$-topes}.
Note that, $ \KKcs , \PPcs $ and $ \PPncs $ are compact spaces.

We have the following natural homeomorphism
\begin{eqnarray*}
  \mathfrak{h}:\KK & \to & \R^d \times (0,\infty) \times \KKcs\\
  K & \mapsto & \left(\cent (K),\Phi(K),\shape (K)\right).
\end{eqnarray*}
Restricting the domain of $\mathfrak{h}$ to $\PP_{n}$ or $\PPnc$ induces the homeomorphisms
\[
  \mathfrak{h}_n :
  P  \mapsto  \left(\cent (P),\Phi(P),\shape (P)\right),
  \ \text{ and } \ 
  \mathfrak{h}_{n,\cent}:
  P  \mapsto  \left(\Phi(P),\shape (P)\right).
\]

The measure $\mu$ given in \eqref{def:Thetamu} is by definition homogeneous and translation invariant.
It gives rise to the measure
\[
  \MeasHalfSpaceMu := \mu \otimes \MeasCount 
\]
on the set of halfspaces $\HS$, where $\MeasCount$ is the counting measure on $ \{ \pm 1 \} $.
This in turn  induces naturally the measure $\MeasPPnMu$ on $\PPn$ via
\begin{equation}
  \label{def:MeasPPnMu}
  \MeasPPnMu (D)
  := \frac{1}{n!} \int\limits_{\HS^n} 
  \1 \left( \bigcap_{i=1}^n H_i^{\epsilon_i} \in D \right) \MeasMultiSimple{\MeasHalfSpaceMu}{n}{H^\epsilon} ,
\end{equation}
where $ \MeasHalfSpaceMu^n := \MeasHalfSpaceMu \otimes \cdots \otimes \MeasHalfSpaceMu $ denotes the product measure and $ \bd{H^\epsilon} := ( H_1^{\epsilon_1} , \ldots , H_n^{\epsilon_n} ) $.
Note that, since $D\subset\PPn$, the integrand $\1 \left( \bigcap_{i=1}^n H_i^{\epsilon_i} \in D \right)$ above is equal to $0$ as soon as the intersection $\bigcap_{i=1}^n H_i^{\epsilon_i}$ is not a polytope with $n$ facets.

Let $A\subset\R^d$ and $C\subset\PPncs$ be Borel sets and $b>0$.
Because the measure $ \MeasPPnMu $ is homogeneous of degree $n$ and translation invariant we obtain
\begin{align*}
  \mathfrak{h}_n \left( \MeasPPnMu \right) (A \times (0,b) \times C)
  & = b^n \mathfrak{h}_n \left( \MeasPPnMu \right) (b^{-1} A \times (0,1) \times C) \\
  & = \MeasLebesgue(b^{-1} A)\, b^n\, \mathfrak{h}_n \left( \MeasPPnMu \right) ([0,1]^d \times (0,1) \times C) \\
  & = \MeasLebesgue (A)\, b^{n-d}\, \mathfrak{h}_n \left( \MeasPPnMu \right) ([0,1]^d \times (0,1) \times C) .
\end{align*}
For $C=\PPn$ this immediately gives the following useful lemma. Here and in the following for abbreviation we put $P_{[n]}= \bigcap_{i\in I} H_i^{\epsilon_i}$ for $\bd{H^\epsilon} := ( H_1^{\epsilon_1} , \ldots , H_n^{\epsilon_n} ) $. 
\begin{lemma}
  \label{lem:Phichange2}
For any $b>0$ and any Borel set $A \subset \R^d $ we have
  \begin{align*}
    & \int\limits_{\HS^n}
    \1 \left( \cent (P_{[n]}) \in A \right)
    \1 \left( \Phi ( P_{[n]} ) < b \right)
    \1 \left( P_{[n]} \in \PPn \right)
    \MeasMultiSimple{\MeasHalfSpaceMu}{n}{H^\epsilon}
    \\ \notag
    & = \MeasLebesgue(A) b^{n-d}
    \int\limits_{\HS^n}    
    \1 \left( \cent (P_{[n]}) \in [0,1]^d \right)
    \1 \left( \Phi ( P_{[n]} ) < 1 \right)   
    \1 \left( P_{[n]} \in \PPn \right)
    \MeasMultiSimple{\MeasHalfSpaceMu}{n}{H^\epsilon} . 
  \end{align*}
 \end{lemma}
 
To simplify our notation we introduce on $\PPncs$ the normalized push forward measure
\[
  \MeasPPncsMu (\cdot)
  = 
  \mathfrak{h}_n(\MeasPPnMu) \left( [0,1]^d \times (0,1) \times \cdot \right) 
  = \MeasPPnMu \left( \mathfrak{h}_n^{-1} \left( [0,1]^d \times (0,1) \times \cdot \right) \right)
\]
and on $\R_+$ the measure 
\begin{equation*}
  \MeasLebesgueModif{n} (\cdot) = 
  \int\limits_0^{\infty} \1 (t\in\cdot)\, n t^{n-1}\, \dint t,
\end{equation*}
which is homogeneous of degree $n$.
With these notations the pushforward measure $ \mathfrak{h}_n ( \MeasPPnMu ) $ splits into the following product of measures:
\begin{equation}
  \label{eq:MeasureSplitting}
  \mathfrak{h}_n \left( \MeasPPnMu \right)
  = \MeasLebesgue \otimes \MeasLebesgueModif{n-d} \otimes \MeasPPncsMu .
\end{equation}
Because  $ \MeasPPncsMu ( \PPncs )$ is finite, $ \MeasPPncsMu (\cdot) / \MeasPPncsMu ( \PPncs ) $ defines a probability measure on $\PPncs$.

That all the measures mentioned in this section are connected seems already obvious.
This will be made precise in the next section.

\subsection{Complementary Theorem}
The essential backbone of our paper is the Complementary Theorem. 
Similar results have been proved before, see e.g. Miles \cite{Miles71}, M{\o}ller and Zuyev \cite{MollerZuyev96} and Cowan \cite{Cowan06}, but for our purposes we need a very detailed description which we could not find in the literature.
Therefore and for the sake of completeness we state it here explicitly and give a proof.
\begin{theorem}
  \label{thm:ComplementaryThm}
  Let $ n \geq d+1 $ be an integer.
  \begin{enumerate}
    \item For any Borel set $S \in \PPncs$ of shapes we have
      \begin{align}
        \label{eq:ProbaIntegralFormbis}
        &\P ( f(Z) = n\, ,\shape(Z) \in S )
        \\&=  \notag
        \frac{ \gamma^d }{ \IntMos } (n-d)!
        \int\limits_{\PPn}
        \1 \left( \cent (P) \in [0,1]^d \right)
        \1 \left( \Phi(P) <1  \right)
        \1 \left( \shape(P) \in S  \right)
        \dint \MeasPPnMu ( P ).
      \end{align}
    \item \textbf{(Complementary Theorem)}
      If we condition the typical cell $Z$ to have $n$ facets, then
      \begin{enumerate}
        \item $ \Phi (Z) $ and $ \shape (Z) $ are independent random variables, 
        \item $ \Phi (Z) $ is $\Gamma_{\gamma, n-d}$ distributed, and 
        \item $\shape(Z)$ has probability distribution  $\MeasPPncsMu (\cdot) / \MeasPPncsMu ( \PPncs )$. 
      \end{enumerate}
  \end{enumerate}
\end{theorem}

\subsection{Proof of Theorem \ref{thm:ComplementaryThm}}
The number of cells of the mosaic $X$ in a subset $D \subset \PPn$ is
\[
  X(D)
  = \frac{1}{n!}
  \sum_{H_1,\ldots,H_n\in\eta^n_{\neq}}
  \sum_{ \vect{\epsilon} \in\{\pm1\}^n}
  \1 \left(\bigcap_{i=1}^n H_i^{\epsilon_i} \in D\right)
  \1 \left(\eta\cap\Big(\bigcap_{i=1}^n  H_i^{\epsilon_i} \Big)^o=\emptyset\right) , 
\]
because there are $n!$ possibilities of ordering a list of $n$ different halfspaces.
The Slivnyak-Mecke formula, see e.g. [SchneiderWeil08 p.68, Corollary 3.2.3] gives for $ P_{[n]} = \bigcap_{i=1}^n H_i^{\epsilon_i} $ that
\begin{align*}
 \E X(D)
  &=
  \frac{\gamma^n}{n!}
  \int\limits_{\HS^n}
  \1 \left( P_{[n]} \in D\right)
  \P \left(\eta \cap P_{[n]}^o =\emptyset\right)
  \MeasMultiSimple{\MeasHalfSpaceMu}{n}{H^\epsilon}
  \\ &=
  \gamma^n\int\limits_{\PPn}
  \1 \left( P \in D\right)
  \P \left(\eta \cap P^o =\emptyset\right)
  \dint \MeasPPnMu(P)
\end{align*}
by the definition of $\MeasPPn$.
Because $\eta$ is a Poisson process we have 
\[
  \P \left(\eta \cap P^o =\emptyset\right)
  = e^{ -\gamma \Phi(P) } .
\] 
In the following we are interested in the case where $D=\mathfrak{h}_n^{-1} ( [0,1]^d \times B \times C)$ with Borel sets $B \subset [0, \infty)$ and $C \subset \PPncs $.
By \eqref{eq:MeasureSplitting} we obtain in this case
\begin{align}
 \notag
 \E X(\mathfrak{h}_n^{-1} ( [0,1]^d \times B \times C ))
 &= \gamma^n\int\limits_{\PPn} \1 ( P \in \mathfrak{h}_n^{-1} ( [0,1]^d \times B\times C )) e^{ - \gamma \Phi(P) } \dint \MeasPPnMu (P)
 \\&= \label{eq:EXsplit}
 \gamma^n\int\limits_{C} \int\limits_B \int\limits_{[0,1]^d}
   \dint \MeasLebesgue(\bd{c}) \  e^{ - \gamma t } \, \dint \MeasLebesgueModif{n-d} (t) \dint \MeasPPncsMu (P) .
\end{align}

\medskip
For the first part of Theorem~\ref{thm:ComplementaryThm}, observe that by the definition \eqref{eq:intSplit} of the intensity measure $\Q$ and using \eqref{eq:EXsplit} we have
\begin{align*}
  \P ( f(Z)=n,\,  \shape(Z) \in C  )
  & = \frac{\IntMos}{\IntMos} \MeasLebesgue ([0,1]^d) \  \Q ( \PPnc \cap \shape^{-1}(C) )\\
  & = 
  \frac1{\IntMos} \E X ( \{ P \in \PPn,\, \cent(P) \in [0,1]^d ,\, \shape(P) \in C \})\\
  & =
  \frac{\gamma^n}{\IntMos}  
  \int\limits_{C} \int\limits_0^\infty \int\limits_{[0,1]^d}
   \dint \MeasLebesgue(\bd{c}) \  e^{ - \gamma t } \, \dint \MeasLebesgueModif{n-d} (t) \dint \MeasPPncsMu (P) . 
\end{align*}
Because the integration with respect to $t$ gives $(n-d)! \gamma^{-(n-d)} \MeasLebesgueModif{n-d}([0,1])$ by elementary computations, the right hand side equals
\[
  \frac{\gamma^d}{\IntMos } (n-d)!\ \MeasPPnMu ( \{ P \in \PPn ,\  \cent(P) \in [0,1]^d ,\ \Phi(P) < 1 , \shape(P) \in C \} ) .
\]
by the definition of $\MeasPPnMu $.
This proves the first part of the theorem.
      
Analogously, for the second part we have 
\begin{align*}
  \P ( f(Z) = n,\,  \Phi(Z) \in B ,\, \shape(Z) \in C )
  &=
  \frac{\gamma^n}{\IntMos}  
  \int\limits_C \int\limits_B \int\limits_{[0,1]^d} 
  \dint \MeasLebesgue(\bd{c})\, e^{ -\gamma t } \,  \dint \MeasLebesgueModif{n-d} (t) \dint \MeasPPncsMu (P)
  \\&=
  \frac{\gamma^n}{\IntMos}   (n-d) \MeasPPncsMu ( C ) \int\limits_B e^{- \gamma t} t^{n-d-1} \dint t .
\end{align*}
Thus, if we condition $Z$ to have $n$ facets, we have that $\shape(Z)$  and $\Phi(Z)$ are independent random variables with distribution
\[
  \P ( \shape(Z) \in C \mid f(Z)=n) =  \MeasPPncsMu (C)/ \MeasPPncsMu (\PPnc)
\]
and 
\[
  \P ( \Phi(Z) \in B \mid f(Z)=n)
  = \frac{\gamma^{n-d}}{(n-d-1)!} \int\limits_B e^{ - \gamma t } t^{n-d-1} \dint t .
\]

\section{Polytope approximation}
\label{sec:PolytopeApproximation}

In this section we provide the necessary geometric ingredients for our main results. The tools used here are from convex geometry, in particular approximation of polytopes by polytopes with fewer facets. 

\subsection{General approximation results for polytopes}
\label{subsec:genapprox}
The starting point of this subsection is \cite{ReisnerSchuttWerner01} from Reisner, Sch\"{u}tt and Werner.
Our goal is to show that if a polytope $P$ has many facets, then a good proportion of them have only a tiny influence.
This will be a key ingredient to obtain a recurrence relation between the probabilities $\P (f(Z)=n)$ and $\P (f(Z)=n-1)$ in Theorem \ref{thm:UpperBoundReccurence}.
More precisely, for $I \subset \N$ and a set of halfspaces $H_i^{\epsilon_i}$, $i \in I$, we define
\[
  P_I := \cap_{i\in I} H_i^{\epsilon_i} . 
\]
Throughout the paper we use the notation 
\[
  [n]=\{1, \dots, n\} .
\]
For $j \leq n$ we have $P_{[n]} \subset P_{[n]\setminus\{j\}}$.
We will measure the distance between $P_{[n]}$ and $P_{[n]\setminus\{i\}}$, both with the Hausdorff distance and the difference of $\Phi$-content.
We will show in the crucial Lemma \ref{lem:RSW3} that for a subset $J \subset [n]$ of size at least $n/4$  we have good upper bounds of the distances between $P_{[n]}$ and $P_{[n]\setminus\{j\}}$ for $j \in J$.

We first present Lemmata \ref{lem:RSW} and \ref{lem:RSW2} which are adaptations of results from \cite{ReisnerSchuttWerner01}.
\begin{lemma}
  \label{lem:RSW}
  There exists a constant $ \Cl[univ]{RSW1} >1$, independent of $d$, such that the following holds.  
  For any integer $ m > \Cr{RSW1}^{(d-1)/2} $ and any $ K \subset \KK $, there exists a polytope $Q \supset K$ with $m$ facets such that
  \[
     d_H ( K , Q) < \Cr{RSW1} c_\Phi \Phi(K) m^{-\frac 2{d-1}}.
  \]
\end{lemma}
\begin{proof}
  This is an immediate consequence of Corollary $2.8.$ of \cite{ReisnerSchuttWerner01} which says that 
  \[
    d_H ( K , Q ) < \Cr{RSW1} R(K) m^{-\frac{2}{d-1}} 
  \]
  where $R(K)$ is the circumradius of $K$, i.e. the radius of the smallest ball containing the convex body $K$.
  The lemma follows from 
  \[
    V_1(K) \geq V_1 ([\vect o, R(K) \vect u]) = R(K) 
  \]
  and \eqref{eq:V1RPhi}.
\end{proof}

The next lemma shows that if the convex body itself is a polytope $P_{[n]}= \cap_{ i = 1 }^n H_i^- $, then $Q$ can be taken as the intersection $P_I$ of suitable supporting halfspaces of $P_{[n]}$.
Its proof is similar to the proof of Lemma~4.3 in \cite{ReisnerSchuttWerner01}.
\begin{lemma}
  \label{lem:RSW2}
  There exist constants $ \Cl[univ]{RSW2} $ and $ \Cl[univ]{RSW3} >0$, such that the following holds.
  For any integer $ k > \Cr{RSW2} $ and any simple polytope $ P_{[n]} $ with $ n $ facets, there exists a subset $ I \subset [ n ] $ with $|I| \leq k $ such that
  \[
    d_H ( P_{[n]}, P_I ) 
    < \Cr{RSW3} c_\Phi \Phi ( P_{[n]} ) k^{ - \frac 2{d - 1} } .
  \]
\end{lemma}
\begin{proof}  
  We set $ \Cr{RSW2} := d(\Cr{RSW1}+1)$ where $ \Cr{RSW1} $ is the constant of Lemma \ref{lem:RSW}.
  We apply Lemma \ref{lem:RSW} to $P_{[n]}$ and $m=\lfloor k/d \rfloor > \Cr{RSW1}^{(d-1)/2}$.
  We obtain a polytope $Q\supset P_{[n]}$ with $\lfloor k/d \rfloor$ facets and 
  \[
    \mathrm{d}_H(P_{[n]},Q)
    < \Cr{RSW1} c_\Phi \Phi(P_{[n]}) \left\lfloor \frac kd \right\rfloor ^{- \frac 2{d-1}}   < \Cr{RSW3} c_\Phi \Phi(P_{[n]}) k ^{- \frac 2{d-1}}. 
  \]
  By eventually shifting and rotating the facets of $Q$ slightly, we can assume that each of the facets of $Q$ meets exactly one vertex of $P_{[n]}$ in its interior.
  Let $I$ be the set of indices of facets of $P_{[n]}$ with one vertex in a facet of $Q$.
  Since $P_{[n]}$ is simple, we have 
  \[
    |I| \leq d\, f(Q) = d \left\lfloor \frac kd \right\rfloor \leq k.
  \]
  Finally, we observe that $ P_{[n]} \subset P_I \subset Q $, which implies $ d_H ( P_{[n]} , P_I ) \leq d_H ( P_{[n]} , Q )$.
\end{proof}

The crucial step is to prove that also the $\Phi$-content between $P_{[n]}$ and $P_{[n]\setminus\{j\}} $ is almost the same for $j \in I$.
\begin{lemma}
  \label{lem:RSW3}
  There exist constants $\Cl[univ]{RSW4}$ and $\Cl[univ]{RSW5}$ such that the following holds.
  For any  $ n > \Cr{RSW4} $ and any simple polytope $ P_{[n]} = \cap_{i=1}^n H_i^- $ with $ n $ facets, there exists a subset $ J \subset [ n ] $ of cardinality at least $ n / 4 $ such that for any $ j \in J $ we have
  \begin{equation}
    \label{eq:dHdiff1}
    d_H \left( P_{[n]} , P_{[n]\setminus\{j\}} \right) 
    < \Cr{RSW5} c_\Phi \Phi(P_{[n]}) n^{-\frac 2{d-1}} ,
  \end{equation}
  and 
  \begin{equation}
    \label{eq:Phidiff1}
    \Phi \left( P_{[n]\setminus\{j\}} \right) 
    < \exp \left\{  \Cr{RSW5} c_\Phi n^{-\frac {d+1}{d-1}}\right\} \Phi(P_{[n]}).
  \end{equation}
\end{lemma}
\begin{proof}
  The first part is just a suitable reformulation of Lemma \ref{lem:RSW2}.
  Set $\Cr{RSW4} := 2 \Cr{RSW2} + 4$, and put $ k = n - 2 \lceil n/4 \rceil $ which implies $k \geq \Cr{RSW2}$.
  By Lemma \ref{lem:RSW2} there is a set $ I \subset \{ 1 , \ldots , n \} $ of cardinality $ k $ such that 
  \begin{equation*}
    d_H ( P_{[n]} , P_I )
    < \Cr{RSW3} c_\Phi  \Phi(P_{[n]}) k^{ - \frac 2{d - 1} }     
    \leq (4d)^{-1} \Cr{RSW5} c_\Phi  \Phi(P_{[n]}) n^{-\frac 2{d-1}}
  \end{equation*}
  for a suitable constant $\Cr{RSW5}  $ by the definition of $k$. 
  Hence for any $j \notin I$,
  \begin{equation}
    \label{eq:RSW2proof}
    d_H \left( P_{[n]} , P_{[n]\setminus\{j\}} \right) 
    \leq d_H \left( P_{[n]} , P_I \right) 
    < (4d)^{-1} \Cr{RSW5}  c_\Phi \Phi(P_{[n]}) n^{-\frac 2{d-1}}
  \end{equation}
  which gives \eqref{eq:dHdiff1}.
  It remains to show that, for at least half of the $j$ not in $I$, equation $(\ref{eq:Phidiff1})$ holds as well.
  Set $\delta = (4d)^{-1} \Cr{RSW5}  c_\Phi \Phi(P_{[n]})  n^{-2/(d-1)}$ and
  \[
    U_j = {\rm cl} \{  \vect{u}  \in \S^{d-1}\colon h(P_{ [n]\setminus\{j\} },  \vect{u} ) \neq h(P_{[n]},  \vect{u} ) \} . 
  \]
  Equation $\eqref{eq:RSW2proof}$ implies that, for any $ j \notin I $ and $  \vect{u}  \in \S^{d-1} $, we have
  \[
    0
    \leq h(P_{ [n]\setminus\{j\} },  \vect{u} ) - h(P_{[n]}, \vect{u} ) 
    \leq \delta \1 (  \vect{u}  \in U_j ).  
  \]
  Therefore,
  \begin{align}
    \notag
    \Phi(P_{ [n]\setminus\{j\} }) - \Phi(P_{[n]}) 
    &= \int\limits_{\S^{d-1}} h(P_{[n]\setminus\{j\}} , \vect{u}) - h(P_{[n]} , \vect{u}) \dint \varphi (\vect{u} )\nonumber
    \\&< \label{eq:PhidiffS} 
    \int\limits_{U_j} \delta \dint \varphi (\vect{u} )
    = \delta \varphi(U_j) .
  \end{align}
  We need to estimate the $\varphi$-measure of the set $ U_j $.
  Denote by $\vect{v}_1, \dots \vect{v}_m$ the vertices of the polytope $P$.
  Since the polytope is simple, each vertex is the intersection of precisely $d$ hyperplanes. 
  Denote by $N ( \vect{v}_l )$ the unit vectors in the normal cone of $P$ at $ \vect{v}_l $, i.e.
  \[
    N(\vect{v}_l ) = \{  \vect{u}  \in \S^{d-1}:\ h(P_{[n]},  \vect{u} ) = \vect{v}_l \cdot  \vect{u}  \} . 
  \]
  The essential observation is that
  \[
    U_j = \bigcup_{ \vect{v}_l \in H_j} N( \vect{v}_l ) . 
  \]
  Observe that the sets $N( \vect{v}_l )$ have pairwise disjoint interiors and cover $\S^{d-1}$.
  Thus for almost all $ \vect{u} \in \S^{d-1} $ we have
  \begin{align*}
    \sum_{j=1}^n \1( \vect{u}  \in U_j) 
    &=
    \sum_{j=1}^n \sum_{l=1}^m \1( \vect{v}_l \in H_j) \1( \vect{u}  \in N( \vect{v}_l)) 
    \\&=
    \underbrace{\sum_{l=1}^m  \1( \vect{u}  \in N( \vect{v}_l))}_{=1} \underbrace{ \sum_{j=1}^n  \1( \vect{v}_l \in H_j)}_{=d}
    = d 
  \end{align*}
  This yields $ \sum_{j=1}^n \varphi (U_j)  = d $ and in particular
  \[
    \sum_{j\notin I} \varphi ( U_j)  \leq d  . 
  \]
  This implies that, for at least half of the $j\notin I$, we have
  \[
    \varphi(U_j) \leq d \left(\frac{n-k}{2}\right)^{-1} =  d \left\lceil \frac n4 \right\rceil ^{-1} \leq 4 d n^{-1} .
  \]
  Otherwise we would have at least half of the $j\notin I$ with the reverse inequality and, because $|I|=k = n - 2 \lceil n/4 \rceil$, that would imply 
  \[
    d  \geq \sum_{j\notin I}^n \varphi (U_j) > \frac12 (n-k) \frac{2d}{n-k} = d .
  \]
  Combined with equation $(\ref{eq:PhidiffS})$, it shows that there exists a set $J\subset \{1,\ldots,n\}\setminus I$ of cardinality $(n-k)/2 = \lceil n/4 \rceil$ such that, for any $j\in J$, we have
  \[
    \Phi(P_{[n]\setminus\{j\}})-\Phi(P_{[n]}) 
    < 4d \delta  n^{-1} 
    = \Cr{RSW5}  c_\Phi  n^{-\frac {d+1}{d-1}} \Phi(P_{[n]}) .
  \]
  This implies equation $(\ref{eq:Phidiff1})$.
\end{proof}

\subsection{Approximation with elongation condition}
The starting point of the following considerations is Theorem 1.1 of \cite{Bonnet16}.
For $1\leq i < j \leq d $ and $\epsilon>0$, we say that a convex body $K$ is $(\epsilon \colon i,j)$-elongated when $V_j(K)^{1/j} V_i(K)^{-1/i}<\epsilon$.
When a convex body, or more specifically a polytope, is sufficiently elongated, the approximation results of the previous subsection can be improved.

\begin{lemma}
  \label{lem:approxElong}
  Assume that $1\leq i < j \leq \lceil (d-1)/2 \rceil$.
  There exist positive constants $\Cl[univ]{elong1}$ and $\Cl[univ]{elong2}$, both depending on $i$, $j$ and $d$, such that the following holds.
  For any $\epsilon>0$, any integer $ k \geq \lfloor \Cr{elong1} \epsilon^{-(d-2)} \rfloor$ and any simple polytope $ P_{[n]} = \cap_{i=1}^n H_i^{\epsilon_i} \in \PPn $ with $n \geq  k$ facets and $ V_j(P_{[n]})^{1/j} V_i(P_{[n]})^{-1/i} < \epsilon$, there exists a subset $ J \subset [n] $ with $ |J| \leq k $, such that
  \[
    d_H ( P_{[n]} , P_{J} ) 
    < \Cr{elong2} \epsilon^{\frac 1{2d}} V_1(P_{[n]}) k^{-\frac 2{d-1}} .
  \]
\end{lemma}
\begin{proof}
  This is a useful application of a recent result by Bonnet \cite{Bonnet16}.
  Assume $ 1 \leq i< j \leq \lceil (d-1) /2 \rceil $.
  Then there exist constants $c_{i,j}$ and $n_{i,j}$ (both depending on $d$), such that the following holds.
  For any $ \epsilon>0 $, any $m \geq n_{i,j} \epsilon^{-(d-2)}$, and any convex body $K$ with
  \[
    \frac { V_j(K)^{\frac 1j} } { V_i(K)^{\frac 1i} } < \epsilon,
  \]
  there exists a polytope $Q \supset K$ with at most $m$ facets  satisfying
  \[
    d_H ( K , Q ) 
    < c_{i,j} \,\epsilon^{\frac 1{2d}}  \, V_1(K) \, m^{- \frac 2{d-1}} .
  \]
  Assume that $P_{[n]}$ is a simple polytope with isoperimetric ratio $ V_j(P_{[n]})^{1/j} V_i(P_{[n]})^{-1/i} < \epsilon$ and $f(P_{[n]}) =n >   k \geq dm $ facets with $ m = \lfloor k/d \rfloor > n_{i,j} \epsilon^{-(d-2)} $. 
  Then there exists a polytope $ Q \supset P_{[n]} $ with $ m+1 $ facets and
  \[
    d_H (P_{[n]},Q) 
    < c_{i,j}  \epsilon^{\frac 1{2d}} \, V_1(P_{[n]}) \, (m+1)^{- \frac 2{d-1}}   
    \leq d^{ \frac 2{d-1}}  c_{i,j} \epsilon^{\frac 1{2d}} \, V_1(P_{[n]}) \, k^{- \frac 2{d-1}}   .
  \]
  We can assume that each of the facets of $Q$ meets exactly one vertex of $P_{[n]}$ in its interior.
  Let $J$ be the set of indices of facets of $P_{[n]}$ with one vertex in a facet of $Q$.
  Since $P_{[n]}$ is simple, we have 
  \[
    |J| \leq d\, f(Q)
    \leq k .
  \]
  And $ P_J \subset Q $ implies 
  \[
    d_H ( P_{[n]} , P_J )
    \leq d_H ( P_{[n]} , Q ) .
  \]  
\end{proof}

In the following lemma we prove the uniform continuity of the isoperimetric ratio.
To our surprise we could not find any results in this direction, this seems to be an open problem.
We state the partial solution to this problem which we need for our purposes.
\begin{lemma}
  \label{lem:ContinuityIsop}
  Let $ 1 \leq i < j \leq d$.
  There exists a constant $\Cl[univ]{contisopratio}$ such that 
  for any  $ \delta \in (0,1) $ and for any $ K,L \in \KK $ with $K \subset L $ and $ d_H(K,L) < \delta V_1(K) $, we have
  \[
    \frac{ V_j(L)^{\frac 1j}} {V_i(L)^{\frac 1i}} < \frac{V_j(K)^{\frac 1j}} {V_i(K)^{\frac 1i}} + \Cr{contisopratio} \delta^{\frac{j-i}{ij(j-1)}} .
  \]
\end{lemma}
\begin{proof}
  A first easy bound is obtained using $ V_i(L)^{j-1} \geq c_{ij} V_1(L)^{j-i} V_j(L)^{i-1}$ which is a consequence of the Alexandrov-Fenchel inequality, see \cite{Schn42nd}, p.401, (7.66) therein.
  \begin{equation}
    \label{eq:AF-inequ}
    \frac {V_j(L)^{\frac 1j}}{V_i(L)^{\frac 1i}} 
    \leq \Cl[univ]{AFinequ} \ \left(\frac{V_i(L)^{\frac 1i}}{V_1(L)} \right)^{\frac{j-i}{j(i-1)}} 
    < \frac{V_j(K)^{\frac 1j} }{ V_i(K)^{\frac 1i}} +
    \Cr{AFinequ} \ \left(\frac{V_i(L)^{\frac 1i}}{V_1(L)} \right)^{\frac{j-i}{j(i-1)}} 
  \end{equation}
  A more precise bound uses Steiner's formula. 
  Due to the isoperimetric inequality \eqref{eq:isopinequ}, $V_i(K)^{1/i} \leq c_{1,i} V_1(K)$ with $c_{1,i} :=  V_i(\B)^{1/i} V_1(\B)$.
  Since 
  \[
    d_H(K,L) < \delta V_1(K) , 
  \]
  we have that $L \subset K + \delta V_1(K) \B$.
  The monotonicity of the intrinsic volumes and Steiner's formula \eqref{eq:steinerintrinsic} shows for $\delta \leq 1$
  \begin{align*}
    V_j (L) 
    &< V_j \left( K + \delta V_1(K)\B \right)
    \\&\leq
    V_j(K) + \sum_{i=1}^j  \binom{d-j+i}{i} \frac{\kappa_{d-j+i}}{\kappa_{d-j}}  c_{1,j-i}^{j-i} V_1(K)^{j-i}(\delta V_1(K))^i  
    \\&\leq
    V_j(K) + \delta V_1(K)^j \sum_{i=1}^j  \binom{d-j+i}{i} \frac{\kappa_{d-j+i}}{\kappa_{d-j}}  c_{1,j-i}^{j-i}
    \\&\leq
    V_j(K) + \Cl[univ]{steinerphi}  \ \delta V_1(L)^j .
  \end{align*}
  Because $ a+b  \leq (a^{\frac 1j}+ b^{\frac 1j})^j $ for $a,b>0$, and because of the monotonicity of the intrinsic volumes this yields
  \begin{equation}
    \label{eq:steiner1}
    \frac{V_j(L)^{\frac 1j}}{V_i(L)^{\frac 1i}}
    \leq 
    \frac{V_j(K)^{\frac 1j} }{ V_i(K)^{\frac 1i}} +
    \Cr{steinerphi}^{\frac 1j}  \delta^{\frac 1j} \frac{V_1(L) }{ V_i(L)^{\frac 1i}} .
  \end{equation}
  Note that $\min \{x, x^{- (j-i)/(j(i-1))} \} \leq 1$ for all $x >0$.
  We define $\Cr{contisopratio}=\max\{\Cr{AFinequ} ,\Cr{steinerphi}^{1/j} \} $ and combine \eqref{eq:AF-inequ} and \eqref{eq:steiner1}.
  \begin{align*}
    \frac{V_j(L)^{\frac 1j}}{V_i(L)^{\frac 1i}}
    &\leq
    \frac{V_j(K)^{\frac 1j} }{ V_i(K)^{\frac 1i}} +   
    \delta^{\frac{j-i}{ij(j-1)}}
    \min\left\{ \Cr{steinerphi}^{\frac 1j}  \delta^{\frac{i-1}{i(j-1)}} \frac{V_1(L) }{ V_i(L)^{\frac 1i}},\ 
    \Cr{AFinequ} \left(\delta^{\frac{i-1}{i(j-1)}} \frac{V_1(L)}{V_i(L)^{\frac 1i} } \right)^{-\frac{j-i}{j(i-1)}}  \right\}
    \\&\leq
    \frac{V_j(K)^{\frac 1j} }{ V_i(K)^{\frac 1i}} +
    \Cr{contisopratio}   
    \delta^{\frac{j-i}{ij(j-1)}}
  \end{align*}
\end{proof}

Recall that we use the notation $P_I = \cap_{i\in I} H_i^{\epsilon_i}$, for any set of integers $I$.
For integers $k\leq n$ and a permutation $\sigma\in\mathfrak{S}_n$ we write $\sigma [k] = \{ \sigma(i) :\  i\in[k] \}$.
In particular $P_{\sigma[k]} = \cap_{i\in I} H_{\sigma(i)}^{\epsilon_{\sigma(i)}}$.
We call hyperplanes $H_i$ in generic position, if the intersection of any $d+2$ of them is empty.
The constants $\Cr{elong1} $ and $ \Cr{elong2}$ have been defined in Lemma \ref{lem:approxElong}.
\begin{lemma}
  \label{lem:sequenceFacets}
  There is a constant $\Cl[univ]{sequfac}$ depending on $\varphi$ such that for all integers  $ i < j \leq \lceil (d-1)/2 \rceil$ and any 
  $ \epsilon  < \Cr{elong1}^{2/(d-1)}\Cr{elong2}^{-1} (c_{\Phi})^{-1}   $
  the following holds.
  For any polytope $ P_{[n]} \in \PPn $  with $n > m=\lfloor \Cr{elong1} \epsilon^{-(d-2)} \rfloor $ facets in generic position and $ V_j(P_{[n]})^{1/j} V_i(P_{[n]})^{-1/i} < \epsilon $ 
  there exist at least $2^{-n} (n-2m)!$ permutations $\sigma \in \mathfrak{S}_n$ such that 
  \begin{enumerate}[(1)]
    \item
      $ d_H ( P_{\sigma[k]} , P_{\sigma[k-1]} )  
      < \Cr{sequfac} c_{\Phi} \, \epsilon^{\frac 1{2d^4}} \, \Phi(P_{\sigma[m]}) k^{-\frac 2{d-1}}$
      for all $k = 2m+1, \dots , n$,
    \item
      $ \| \cent(P_{\sigma[n]}) - \cent(P_{\sigma[m]}) \| < \Phi(P_{\sigma[n]}) $, and
    \item
      $ \Phi(P_{\sigma[m]}) < 2 \Phi(P_{\sigma[n]}) $.
  \end{enumerate}
\end{lemma}
\begin{proof}
  We set 
  \[
    m =  \lfloor \Cr{elong1} \epsilon^{-(d-2)} \rfloor .
  \]
  By Lemma \ref{lem:approxElong} there exists a subset $ I \subset [n] $ with $ |I| =m $, such that for all subsets $J$ with 
  $ I \subset J \subset [n]$ we have
  \[
    d_H ( P_{[n]} , P_{J} ) 
    < \Cr{elong2} \epsilon^{\frac 1{2d}} V_1(P_{[n]}) m^{-\frac 2{d-1}} 
    < \Cr{elong1}^{-\frac 2{d-1}}\Cr{elong2} \epsilon \, V_1(P_{[n]}) .
  \]
  By Lemma \ref{lem:ContinuityIsop} this implies for all such sets $J$ that
  \begin{equation}
    \label{eq:unifisopratiobd}
    \frac{ V_j(P_J)^{\frac 1j}} {V_i(P_J)^{\frac 1i}} 
    < \e  + \Cr{contisopratio} (\Cr{elong1}^{-\frac 2{d-1}}\Cr{elong2} \epsilon)^{\frac{j-i}{ij(j-1)}} 
    < \Cl[univ]{elong3} \epsilon^{\frac 1{d^3}} .
  \end{equation}  
  We denote by $S(P_{[n]}) \subset \mathfrak{S}_n$ the set of those permutations $\sigma$ such that
  \begin{enumerate}[(a)]
   \item  $\sigma[m]=I$, and 
   \item $ d_H ( P_{\sigma[k]} , P_{\sigma[k-1]} )  
     < 2^{\frac 2{d-1}} \Cr{elong2} \Cr{elong3}^{\frac 1{2d}} \epsilon^{\frac 1{2d^4}} V_1(P_{\sigma[k]}) k^{-\frac 2{d-1}}$
     for all $k =n, \dots , 2m+1$.
  \end{enumerate}
  To estimate $|S(P_{[n]})|$ note first that there are $m!$ possibilities such that $\sigma[m]=I$. Second, assume that 
  $\sigma(n), \dots, \sigma(k+1) \in [n] \setminus I$ are already chosen satisfying Condition (b).
  Then by \eqref{eq:unifisopratiobd} and by Lemma \ref{lem:approxElong} applied to the polytope $P'=P_{\sigma[k]}$, the integer $k'=\frac{k}{2}\ge m$ and $\epsilon'=\Cr{elong3} \epsilon^{\frac 1{d^3}}$, there is a set $J_k \subset \sigma[k]$ of size $|J_k| \leq k/2$ such that
  \[
    d_H ( P_{\sigma[k]} , P_{J_k} ) 
    < \Cr{elong2} \Cr{elong3}^{\frac 1{2d}} \epsilon^{\frac 1{2d^4}} V_1(P_{\sigma[k]}) \left( \frac k2\right)^{-\frac 2{d-1}} .
  \]
  If we choose $\sigma(k) \notin J_k$, Condition (b) is thus satisfied. Because we need in addition $\sigma(k) \notin I$ there are at least $k/2 - m$ possibilities to choose $\sigma(k)$,
  and thus to determine $\sigma[k-1]$. 
  Continuing until $k=2m+1$ gives at least $\prod_{2m+1}^n (k/2-m) $ possibilities to choose $\sigma(n), \dots, \sigma(2m+1)$.
  We obtain
  \[
    |S(P_{[n]})| \geq m! \prod_{k=2m+1}^n \left(\frac k2 -m\right)
    = m! 2^{-n+2m} (n-2m)! > 2^{-n} (n-2m)!
  \]
  Using \eqref{eq:V1RPhi}, we observe that Condition (1) of our lemma is satisfied by choosing $\Cr{sequfac} = 2^{\frac 2{d-1}} \Cr{elong2} \Cr{elong3}^{\frac 1{2d}} $ in Condition (b).
  Condition (2) follows from  the 1-Lipschitz property of $\cent$ and
  \[
    d_H ( P_{[n]} , P_{\sigma[m]} ) 
    < \Cr{elong1}^{-\frac 2{d-1}}\Cr{elong2} c_{\Phi}  \Phi(P_{[n]})   \epsilon  < \Phi(P_{[n]}) ,
  \]
  if $ \Cr{elong1}^{-\frac 2{d-1}}\Cr{elong2} c_{\Phi}   \epsilon  < 1 $.
  Condition (3) follows from this and the fact that 
  \[
    \Phi(P_{\sigma[m]}) < \Phi(P_{\sigma[n]} + \Phi(P_{\sigma[n]}) B^d) < 2 \Phi(P_{\sigma[n]}) .
  \]
\end{proof}

\section{Cells with many facets}
\label{sec:manyfac}

The aim of this section is to show how Theorem \ref{thm:ComplementaryThm}, combined with the geometric arguments developed in Section~\ref{sec:PolytopeApproximation}, implies our main results, i.e. Theorems \ref{mainthm:UpperBoundfn}, \ref{mainthm:LowerBoundfn} and \ref{mainthm:boundsManyFacetsIsoperimeter}.
To do so, we start by presenting the three intermediary results that will play a key role in the proofs of these theorems. 

By seeing a polytope with $ n $ facets as a polytope with $ (n - 1) $ facets cut `a little bit' by one halfspace, we obtain the following recurrence relation. 
\begin{theorem}
  \label{thm:UpperBoundReccurence}
  There exist constants $\Cl[univ]{decreases1}$ and $\Cl[univ]{decreases2}$, independent of $\varphi$, such that for $n > (\Cr{decreases1} c_\Phi )^{d/2}$,
  \[
    \P ( f(Z)=n )
    \leq \Cr{decreases2} c_\Phi n^{-\frac 2{d-1}} \P ( f(Z)=n-1)
  \]
  and $\Cr{decreases2} c_\Phi n^{-2/(d-1)} < 1$, where 
  $c_\Phi$ is defined in \eqref{eq:V1RPhi}.
\end{theorem}
\bigskip

The next theorem provides an upper-bound for the probability of the same event $\{f(Z)=n\}$ intersected with the event of being  $(\epsilon \colon i,j)$-elongated.
Let $\Cr{elong1}$, $\Cr{elong2}$, and $c_{\Phi}$ be the constants used in Lemma \ref{lem:sequenceFacets}.

\begin{theorem}
  \label{thm:boundsManyFacetsIsoperimeter}
  Assume $1 \leq i < j \leq \lceil (d-1) / 2 \rceil $. There exist constants $\Cl[univ]{IntPolyM}$, $\Cl[univ]{thmelong1}$ and $\Cl[univ]{thmelong2}$, such that for any 
  $ \epsilon  < \Cr{elong1}^{2/(d-1)}\Cr{elong2}^{-1} (c_{\Phi})^{-1}   $ we have 
  \begin{equation*}
    \P \left( f(Z)=n, \ \frac{V_j(Z)^{\frac 1j}}{V_i(Z)^{\frac 1i}} <  \epsilon \right)
    <   \frac{ \gamma^d }{ \IntMos } 
    \ e^{ \Cr{thmelong1} \epsilon^{-2(d-1)}} (\Cr{thmelong2}\epsilon^{\frac 1{2d^4}})^n \  n^{-\frac{2n}{d-1}}  
  \end{equation*}
  for $n > \lfloor \Cr{elong1} \epsilon^{-(d-2)} \rfloor^2 $.
\end{theorem}
As we will see, the bound in Theorem \ref{thm:boundsManyFacetsIsoperimeter}, i.e. when we add the condition that the isoperimetric ratio is small,  is close to the one that we will get when iterating Theorem \ref{thm:UpperBoundReccurence} but with a constant in front of $n^{-\frac{2n}{d-1}}$ which is arbitrarily small.

The last theorem deals with the lower-bound for $\P \left(f(Z)=n\right)$.
This requires an extra-condition on the directional distribution $\varphi$.
Recall that we call $\varphi$ {\it well spread} if there exists a cap on the unit sphere where $\varphi$ is bounded from below by a multiple of the surface area measure.
We denote by $\mathscr{H}^{d-1} ( \cdot ) $ the $(d-1)$-dimensional Hausdorff measure, the surface area.
We will prove a slightly more precise form of Theorem \ref{mainthm:LowerBoundfn}
\begin{theorem}
  \label{thm:MainLowerBound}
  There exist constants $ \Cl[univ]{LowerBound1} > 0$, $\Cl[univ]{BoundedPoly1}\in\N$, independent of $\varphi$,  such that the following holds.
  Assume that $\varphi$ is well spread.
  In particular assume that there exists a cap $ C \subset \S^{d-1} $ of radius $ r \in (0,1) $ and a constant 
  $ \Cl[univ]{const:5} $ with $ \varphi ( \cdot ) > \Cr{const:5} \mathscr{H}^{d-1} ( \cdot ) $ on $C$.
  Then, for $ n \geq \Cr{BoundedPoly1} $, we have 
  \[
    \P ( f( Z)=n )
    > \frac{\gamma^d}{\IntMos}
    ( \Cr{const:5} \Cr{LowerBound1} r^{d+2} )^n
    n^{-\frac{2n}{d-1}}.
  \]
\end{theorem}

In the next subsection, we show how to deduce in a very small number of steps our main results from the three theorems above.
The rest of Section \ref{sec:manyfac} is devoted to the proof of Theorems \ref{thm:UpperBoundReccurence}, \ref{thm:boundsManyFacetsIsoperimeter} and \ref{thm:MainLowerBound}.

\subsection{Deducing Theorems \ref{mainthm:UpperBoundfn}, \ref{mainthm:LowerBoundfn} and \ref{mainthm:boundsManyFacetsIsoperimeter} from Theorems \ref{thm:UpperBoundReccurence}, \ref{thm:MainLowerBound} and \ref{thm:boundsManyFacetsIsoperimeter}}
Set $n_0 := \lceil (\Cr{decreases1} c_\Phi )^{d/2} \rceil $ where $\Cr{decreases1}$ is given by Theorem \ref{thm:UpperBoundReccurence}.
Iterating Theorem \ref{thm:UpperBoundReccurence}, gives us that for any $n\geq n_0$,
\begin{equation*}
  \P( f(Z)=n) 
  \leq (\Cr{decreases2} c_\Phi)^{n-n_0} \left(\frac{n!}{n_0!}\right)^{-\frac 2{d-1}}.
\end{equation*}
Now Stirling's approximation  $ n! > n^{n} e^{-n} $ 
implies  for any $n \geq n_0$,
\[
  \P ( f(Z)=n )
  < (\Cr{decreases2} c_\Phi)^{-n_0} \left(n_0!\right)^{\frac2{d-1}} (e^{\frac 2{d-1}} \Cr{decreases2} c_\Phi)^{n}  n^{-\frac{2n}{d-1}}
\]
which implies Theorem \ref{mainthm:UpperBoundfn}.

\medskip
Taking 
$
  \Cr{lowerbound2TH2} 
  =\left( \min(1,\frac{\gamma^d}{\IntMos}) ^{\frac{1}{\Cr{BoundedPoly1}}} \right) \Cr{const:5} \Cr{LowerBound1} r^{d+2}
$
where $\Cr{BoundedPoly1}$, $\Cr{const:5}$, $\Cr{LowerBound1}$ and $r$ are given by Theorem \ref{thm:MainLowerBound}, we obtain  Theorem \ref{mainthm:LowerBoundfn}.

\medskip
Taking 
$
  \Cl[univ]{thm1-3-1bis}
  =\Cr{elong1}^{2/(d-1)}\Cr{elong2}^{-1} (c_{\Phi})^{-1}
$,
$
  \Cl[univ]{thm1-3-2bis}
  =\frac{\Cr{thmelong2}\epsilon^{\frac 1{2d^4}}}{\Cr{const:5} \Cr{LowerBound1} r^{d+2}}
$
and 
$
  \Cl[univ]{depepsilonbis}^{(\epsilon)}
  =\ e^{ \Cr{thmelong1} \epsilon^{-2(d-1)}}
$,
we deduce from Theorem \ref{thm:MainLowerBound} and Theorem \ref{thm:boundsManyFacetsIsoperimeter}, when $\varphi$ is well spread, that
\[ 
  \P \left( \frac{V_j(Z)^{\frac 1j}}{V_i(Z)^{\frac 1i}}< \epsilon \  \Big\vert f(Z)=n\right) 
  \leq  
  {\Cr{depepsilonbis}}^{(\epsilon)}(\Cr{thm1-3-2bis} \epsilon^{\frac 1{2d^4}})^n  
\]
for any $\epsilon < \Cr{thm1-3-1bis} $ and $n \geq \max \left( \Cr{BoundedPoly1} , \lfloor \Cr{elong1} \epsilon^{-(d-2)} \rfloor^2 \right) $.
Now, choose $\epsilon$ such that $ \Cr{thm1-3-2bis} \epsilon^{\frac 1{2d^4}}=\delta/2 $.
Theorem \ref{mainthm:boundsManyFacetsIsoperimeter} follows from the fact that $ {\Cr{depepsilonbis}}^{(\epsilon)} \left(\frac{\delta}{2}\right)^n < \delta^n $, for $ n > \ln( {\Cr{depepsilonbis}}^{(\epsilon)} ) / \ln (2) $.

\subsection{Proof of Theorem \ref{thm:UpperBoundReccurence}}
\label{section:distribqn}

We first need to state the following elementary but useful lemma.
We denote by $ \mathfrak{S}_n $ the set of permutations of $[n]$.
For $  \vect{x} = ( x_1 , \ldots , x_n ) $ and $ \sigma \in  \mathfrak{S}_n $, we write
$  \vect{x}_\sigma := ( x_{\sigma(1)} , \ldots , x_{\sigma(n)} ) $.
It is clear that the following holds.
\begin{lemma}
  \label{lem:TechnicalPermutation}
  Let $ ( X , \Sigma , \psi ) $ be a measured space,
  $ m,n > 0 $ be integers,
  $ f : X^n \to [0,\infty) $ be a measurable function
  and $ S , T \subset X^n $ measurable sets.
  Assume that
  \begin{itemize}
    \item $ f $ is symmetric:
      for any $ \sigma \in \mathfrak{S}_n $
      and any $  \vect{x} \in X^n $,
      we have $ f (  \vect{x}_\sigma ) = f (  \vect{x} ) $;
    \item $ S $ is symmetric:
      for any $ \sigma \in \mathfrak{S}_n $,
      and any $  \vect{x} \in X^n $ we have
      $ \1 (  \vect{x}_\sigma \in S ) = \1 (  \vect{x} \in S ) $;
    \item for any $  \vect{x} \in S $,
      there exist at least $ p $ permutations $ \sigma \in \mathfrak{S}_n $
      such that $ \vect{x}_\sigma \in T $.
  \end{itemize}
  Then
  \[
    \frac{p}{n!}
    \int\limits_{ X^n }
    \1 (  \vect{x} \in S )
    f (  \vect{x} ) 
    \MeasMultiSimple{\psi}{n}{x}
    \leq 
    \int\limits_{ X^n } 
    \1 ( \vect{x} \in T )
    f ( \vect{x} )
    \MeasMultiSimple{\psi}{n}{x} .
  \]
\end{lemma}

The next lemma deals with the measure of those polytopes $P_{[n]}$ which are close to $P_{[n-1]}$ in the Hausdorff distance.
\begin{lemma}
  \label{lem:IntTakeOutFacets}
    For any $\e>0$ and any measurable function $f:\HS^{n-1}\to (0,\infty)$, it holds that
    \begin{multline}
      \label{eq:ineq2}
      \int\limits_{\HS^n}
      \1 \left( P_{[n]} \in \PPn \right)
      \1 \left( d_H ( P_{[n]} , P_{[n-1]} ) < \e \right)
      f ( H_1^{\epsilon_1} , \ldots , H_{n-1}^{\epsilon_{n-1}} )    
      \MeasMultiSimple{\MeasHalfSpaceMu}{n}{H^\epsilon}
      \\
      < \e
      \int\limits_{\HS^{n-1}}
      \1 \left( P_{[n-1]} \in \PPn[n-1] \right)
      f ( \bd{H^\epsilon} )    
      \MeasMultiSimple{\MeasHalfSpaceMu}{n-1}{H^\epsilon} .
    \end{multline}
\end{lemma}
\begin{proof}
  In this particular proof, we use the following representation of half-spaces: for any $\vect{u}\in\S^{d-1}$ and $t\in\R$, we denote by $\widetilde{H}(u,t)=\{ \vect{x}\in\R^d :\  \langle \vect{x}, \vect{u} \rangle\le t\}$ so that 
  \[
    \MeasHalfSpaceMu(\cdot)
    =\int\limits_{\S^{d-1}} \int\limits_{\R} \1 \left(\widetilde{H}(u,t)\in \cdot \right) \dint t \dint \varphi (\vect{u} ).
  \]
  For any $ K \in \KK  $, we now proceed with the following calculation:
  \begin{align}
    \label{eq:ineqinterm}
    &\int\limits_{\HS} \1 \left( K \cap H_n \neq \emptyset \right) \1 \left( d_H ( K , K\cap H_n^{\epsilon_n}) < \e  \right) \dint \MeasHalfSpaceMu (H_n^{\epsilon_n})
    \notag\\
    & = \int\limits_{\S^{d-1}} \int\limits_{\R} 
    \1 \left( K \cap H( \vect{u} , t )  \neq \emptyset \right) \1 \left( d_H ( K , K\cap \widetilde{H}( \vect{u} , t ) ) < \e ) \right)     \dint t \dint \varphi (\vect{u} )
    \notag\\
    & \leq \int\limits_{\S^{d-1}} \int\limits_{h(K,\vect{u})-\e}^{h(K,\vect{u})}
    \dint t \dint \varphi (\vect{u} ) 
    = \e .
  \end{align}
  Let us fix now $\vect{H^\epsilon}\in \HS^{n-1}$. We observe that for every $H_n^{\epsilon}\in \HS$,
  \begin{align}
    \label{eq:ineqinterm2}
    &\hspace*{0cm}\1 \left( P_{[n]} \in \PPn \right)
    \1 \left( d_H ( P_{[n]} , P_{[n-1]} ) < \e \right)
    \nonumber\\
    &\leq \1 \left( P_{[n-1]} \in \PPn[n-1] \right) \1 \left( P_{[n-1]} \cap H_n \neq \emptyset \right) \1 \left( d_H ( P_{[n-1]} , P_{[n-1]}\cap H_n^{\epsilon_n}) < \e  \right).
  \end{align}
  Integrating \eqref{eq:ineqinterm2} over $H_n^{\epsilon_n}\in \HS$ and combining it with \eqref{eq:ineqinterm} applied to $K=P_{[n-1]}$, we obtain
  \[
    \int\limits_{\HS} \1 \left( P_{[n]} \in \PPn \right)
    \1 \left( d_H ( P_{[n]} , P_{[n-1]} ) < \e \right)
    \dint \MeasHalfSpaceMu (H_n^{\epsilon_n})
    \le \e \1 \left( P_{[n-1]} \in \PPn[n-1] \right).
  \]
  We conclude by multiplicating the previous inequality by $f ( H_1^{\epsilon_1} , \ldots , H_{n-1}^{\epsilon_{n-1}} )$ and integrating it  with respect to $( H_1^{\epsilon_1} , \ldots , H_{n-1}^{\epsilon_{n-1}} )\in\MeasMultiSimple{\MeasHalfSpaceMu}{n-1}{H^\epsilon}$.
\end{proof}

\bigskip
We are now in the position to prove Theorem \ref{thm:UpperBoundReccurence}. 
Set $ \alpha = \Cr{RSW5} c_\Phi n^{-2/(d-1)} $, and set 
  \[
    I_n = \frac{ \IntMos }{ \gamma^d } \frac{n!}{(n-d)!} \P ( f(Z) = n) .
  \]
  By \eqref{def:MeasPPnMu} and \eqref{eq:ProbaIntegralFormbis}, we have 
  \[
    I_n = \int\limits_{\HS^n}
    \1 \left( P_{[n]} \in \PPn \right) 
    \1 \left( \cent (P_{[n]}) \in [0,1]^d \right)
    \1 \left( \Phi(P_{[n]}) <1  \right)
    \MeasMultiSimple{\MeasHalfSpaceMu}{n}{H^\epsilon} ,
  \]
  where $P_{[n]} = \cap_{i=1}^n H_i^{\epsilon_i}$.
  We want to use now Lemma $\ref{lem:RSW3}$ which, roughly speaking, tells us that the variable $H_n$ has a `small influence'.
  Set 
  \[
    S = \{ (  \vect{H}  ,  \vect{\epsilon}  ) \in \HS^n :\  \cap_{i=1}^n H_i^{\epsilon_i} \in \PPn 
    \text{ is a simple polytope}\},
  \]
  and 
  \[
    T = \{ (  \vect{H}  ,  \vect{\epsilon}  ) \in S :\ 
    d_H \left( P_{[n]} , P_{[n-1]} \right) < \alpha \Phi(P_{[n]}) ,\ 
    \Phi \left( P_{[n-1]} \right) < \exp \left\{ \alpha n^{-1} \right\} \Phi(P_{[n]}) \} .
  \]
  Lemma $\ref{lem:RSW3}$ tells us that, for any $ (\vect{H}  ,  \vect{\epsilon}  ) \in S $, there exists at least $n!/4$ permutations $\sigma\in\mathfrak{S}_n$ such that $(\vect{H}  ,  \vect{\epsilon}  )_\sigma \in T$.
  Hence, Lemma $\ref{lem:TechnicalPermutation}$ and the Lipschitz continuity of $\cent$ imply
  \begin{align*}
    \frac{I_n}4  
    &\leq  
      \int\limits_{\HS^n}
      \1 \left( P_{[n]} \in \PPn \right) 
      \1 \left( \cent (P_{[n]}) \in [0,1]^d \right)
      \\& \phantom{{}\leq\int\limits_{\HS^n}{}}
      \1 \left( d_H \left( P_{[n]} , P_{[n-1]} \right) < \alpha \right)
      \1 \left( \Phi \left( P_{[n-1]} \right) < \exp \left\{ \alpha n^{-1} \right\} \right)
      \MeasMultiSimple{\MeasHalfSpaceMu}{n}{H^\epsilon} 
    \\
    &\leq 
      \int\limits_{\HS^n}
      \1 \left( P_{[n]} \in \PPn[n] \right)
      \1 \left( \cent (P_{[n-1]}) \in [-\alpha,1+\alpha]^d \right)
      \\& \phantom{{} \leq \int\limits_{\HS^n} {}}
      \1 \left( d_H \left( P_{[n]} , P_{[n-1]} \right) < \alpha \right) 
      \1 \left( \Phi \left( P_{[n-1]} \right) < \exp \left\{ \alpha n^{-1} \right\}  \right)
      \MeasMultiSimple{\MeasHalfSpaceMu}{n}{H^\epsilon} .
  \end{align*}  
  Now, using consecutively Lemma \ref{lem:IntTakeOutFacets} and Lemma \ref{lem:Phichange2} applied to an $(n-1)$-fold integral, we obtain that
  \begin{align*}
    \frac{I_n}4
    &\leq
      \alpha \int\limits_{\HS^{n-1}}
      \1 \left( P_{[n-1]} \in \PPn[n-1] \right)
      \1 \left( \cent (P_{[n-1]}) \in [-\alpha,1+\alpha]^d \right)
      \\& \phantom{{} \leq \alpha \int\limits_{\HS^{n-1}} {}}
      \1 \left( \Phi \left( P_{[n-1]} \right) < \exp \left\{ \alpha n^{-1} \right\}  \right)
      \MeasMultiSimple{\MeasHalfSpaceMu}{n-1}{H^\epsilon}
    \\ 
    &\leq 
      \alpha (1+2\alpha)^{d} \exp \left\{ \alpha (n-1-d)n^{-1} \right\}
      \\& \phantom{{} \leq {}} \cdot
      \int\limits_{\HS^{n-1}}
      \1 \left( P_{[n-1]} \in \PPn[n-1] \right)
      \1 \left( \cent (P_{[n-1]}) \in [0,1]^d \right)
      \1 \left( \Phi \left( P_{[n-1]} \right) < 1 \right)
      \MeasMultiSimple{\MeasHalfSpaceMu}{n-1}{H^\epsilon} 
    \\&\leq
      \alpha (1+2\alpha)^d \exp \left\{ \alpha \left(1-\frac{d+1}{n}\right) \right\} I_{n-1} .
  \end{align*}
  Therefore,
  \begin{align*}
   \P ( f(Z) = n) 
   &\leq   
   4\alpha (1+2\alpha)^d \exp \left\{ \alpha \left(1-\frac{d+1}{n}\right)\right\} 
   \frac{(n-d)} {n}  \P ( f(Z) = n-1)
   \\&\leq   
   4\alpha \exp \left\{ 2d\alpha + \alpha \left(1-\frac{d+1}{n}\right) \right\} 
    \P ( f(Z) = n-1)  .
  \end{align*}
  This proves that for   $ n > \Cr{RSW4} $,
  we have 
  \[
    \P ( f(Z)= n) 
    <  4\alpha \exp\{ 3 d \alpha \} \P ( f(Z)= n - 1 ).
  \]
  Set $\Cr{decreases2} := 4 e^{3d} \Cr{RSW5}$ and recall that $\alpha = \Cr{RSW5} c_\Phi n^{-2/(d-1)}$.
  Theorem \ref{thm:UpperBoundReccurence} follows for $ \Cr{RSW5} c_\Phi n^{-2/(d-1)} < 1$, since in that case
  $
    4\alpha e^{ 3 d \alpha } \leq 4 \alpha e^{3d}
    = \Cr{decreases2} c_\Phi n^{-2/(d-1)}
  $.
    
\subsection{Proof of Theorem \ref{thm:boundsManyFacetsIsoperimeter}}

We will proceed in a similar way as in the proof of Theorem \ref{thm:UpperBoundReccurence} with one  main difference: 
in order to take into account the elongation condition, we will use Lemma \ref{lem:sequenceFacets} instead of Lemma \ref{lem:RSW3}.
In particular, Lemma \ref{lem:sequenceFacets} does not guarantee that there is a permutation $\sigma$ such that the  condition on the isoperimetric ratio $\frac{V_j(P_{[n]})^{\frac 1j}}{V_i(P_{[n]})^{\frac 1i}} < \epsilon $ is satisfied by $P_{\sigma[n-1]}$ so there is no possibility to apply it more than once. 
This explains why we have directly a general upper bound but not a recurrence relation similar as the one of Theorem \ref{thm:UpperBoundReccurence}.

Let $\Cr{elong1}$, $\Cr{elong2}$, $\Cr{sequfac}$ and $c_{\Phi}$ be the constants used in Lemma \ref{lem:sequenceFacets}, assume 
$ \epsilon  < \Cr{elong1}^{2/(d-1)}\Cr{elong2}^{-1} (c_{\Phi})^{-1}   $, and set $m=\lfloor \Cr{elong1} \epsilon^{-(d-2)} \rfloor $ and 
\[
  \delta = \Cr{sequfac} c_{\Phi} \epsilon^{\frac 1{2d^4}} . 
\]
Set 
\[
  I_n := 
  \frac{ \IntMos }{ \gamma^d } \frac{n!}{(n-d)!} 
  \P \left( 
    f(Z) =n,\  
    \frac{V_j(Z)^{\frac 1j}}{V_i(Z)^{\frac 1i}} < \epsilon \right) .
\]
By \eqref{eq:ProbaIntegralFormbis}, we have 
\begin{align*}
  & I_n = \\
  & \int\limits_{\HS^n}
  \1 \left( P_{[n]} \in \PPn \right)
  \1 \left( \cent (P_{[n]}) \in [0,1]^d \right)
  \1 \left( \Phi(P_{[n]}) <1  \right)
  \1 \left( \frac{V_j(P_{[n]})^{\frac 1j}}{V_i(P_{[n]})^{\frac 1i}} < \epsilon \right)
  \MeasMultiSimple{\MeasHalfSpaceMu}{n}{H^\epsilon} ,
\end{align*}
where $P_{[n]} = \cap_{i=1}^n H_i^{\epsilon_i}$.
Roughly speaking, we will now use Lemmata \ref{lem:TechnicalPermutation} and \ref{lem:sequenceFacets} to order the halfspaces such that integrating step by step, starting by $H_n^{\epsilon_n}$, the integrals can be well bounded.
Set
\[
  S = \left\{ (  \vect{H}  ,  \vect{\epsilon}  ) \in \HS^n :\  
  P_{[n]} \in \PPn \mbox{ with facets in generic position},\ 
  \frac{V_j(P_{[n]})^{\frac 1j}}{V_i(P_{[n]})^{\frac 1i}}  < \epsilon \right\},
\]
and 
\begin{align*}
T &= \Big\{ (  \vect{H}  ,  \vect{\epsilon}  ) \in \HS^n :\ 
       P_{[n]} \in \PPn , \,  
       \| \cent(P_{[n]}) - \cent(P_{[m]}) \| < \Phi(P_{[n]}),\, 
       \Phi(P_{[m]}) < 2 \Phi(P_{[n]}),\,
   \\& \qquad 
       d_H ( P_{[k]} , P_{[k-1]} ) < \delta \Phi(P_{[m]}) k^{-\frac 2{d-1}} \text{ for } 2m < k \leq n 
     \Big\} .
\end{align*}
Lemma \ref{lem:sequenceFacets} tells us that, for any $ (\vect{H}  ,  \vect{\epsilon}  ) \in S $, there exist at least $2^{-n}(n-2m)!$ permutations $\sigma\in\mathfrak{S}_n$ such that $(\vect{H}  ,  \vect{\epsilon}  )_\sigma \in T$.
Hence, Lemma \ref{lem:TechnicalPermutation} implies
\begin{align*}
    &\frac{2^{-n} (n-2m)!}{n!} I_n
  \\
    &\leq
    \int\limits_{\HS^n}
    \1 \left( P_{[n]} \in \PPn \right)
    \1 \left( \cent (P_{[n]}) \in [0,1]^d \right)
    \1 \left( \Phi(P_{[n]}) <1  \right)
    \\& \phantom{{} \leq \int\limits_{\HS^n} {}}
    \1 \left( \| \cent(P_{[n]}) - \cent(P_{[m]}) \| < \Phi(P_{[n]}) \right)
    \1 \left( \Phi(P_{[m]}) < 2 \Phi(P_{[n]}) \right) 
    \\& \phantom{{} \leq \int\limits_{\HS^n} {}}
    \1 \left( d_H ( P_{[k]} , P_{[k-1]} ) < \delta  \Phi( P_{[m]} ) k^{-\frac 2{d-1}} \text{ for } 2m <  k \leq n \right)
    \MeasMultiSimple{\MeasHalfSpaceMu}{n}{H^\epsilon} 
  \\
    &\leq
    \int\limits_{\HS^n}
    \1 \left( P_{[n]} \in \PPn \right)
    \1 \left( \cent (P_{[m]}) \in [-1,2]^d \right)
    \1 \left( \Phi(P_{[m]}) < 2  \right)
    \\& \phantom{{} \leq \int\limits_{\HS^n} {}} 
    \1 \left( d_H ( P_{[k]} , P_{[k-1]} ) \in ( 0 , 2 \delta k^{-\frac 2{d-1}} ) \text{ for } 2m < k \leq n \right)
    \MeasMultiSimple{\MeasHalfSpaceMu}{n}{H^\epsilon} .
\end{align*}
Now, using $n-2m$ times \eqref{eq:ineq2}, we have
\begin{align*}
  \frac{2^{-n} (n-2m)!}{n!} I_n
  < (2 \delta)^{n-2m} \left( \frac{n!}{(2m)!} \right) ^{-\frac 2{d-1}}
  \Cr{IntPolyM} ,
\end{align*}  
where $ \Cr{IntPolyM} = \Cr{IntPolyM}(m,d,\varphi) $ is defined by
\[
  \Cr{IntPolyM}
  := \int\limits_{\HS^{2m}}
    \1 \left( P_{[2m]} \in \PP_{2m} \right)
    \1 \left( \cent (P_{[m]}) \in [-1,2]^d \right)
    \1 \left( \Phi(P_{[m]}) < 2  \right)
  \MeasMultiSimple{\MeasHalfSpaceMu}{2m}{H^\epsilon} .
\]
This implies for $n >m^2$
\begin{align}
  \notag
  \P \left( f(Z) =n,\ \frac{V_j(Z)^{\frac 1j}}{V_i(Z)^{\frac 1i}} < \epsilon \right) 
  &  < 
  \frac{ \gamma^d }{ \IntMos }   \Cr{IntPolyM} (2 \delta)^{-2m}  ((2m)!) ^{\frac 2{d-1}} 
  n^{2m-d} (4 \delta)^{n} (n!) ^{-\frac 2{d-1}}  
  \\  & \leq  \label{eq:fgepsdef}
  \frac{ \gamma^d }{ \IntMos }   \Cr{IntPolyM} g(\e) f(\e)^n \  n^{-\frac{2n}{d-1}}  
\end{align}
where we defined
\[
  g(\e) 
  := 
  \exp\left\{ \left(\frac {8 \Cr{elong1}^2 }{d-1} + \frac {\Cr{elong1} + 1 }{\Cr{sequfac} c_{\Phi}  }\right) \ \epsilon^{-2(d-1)} )   \right\}
  , \textrm{ and }\  
  f(\e)
  :=
  4 e^{\frac{4}{e}+\frac2{d-1}} \Cr{sequfac} c_{\Phi} \ \epsilon^{\frac 1{2d^4}}.
\]
The estimates in \eqref{eq:fgepsdef} hold because using $n! < n^n$,
\begin{align*}
  (2 \delta)^{-2m}((2m)!) ^{\frac 2{d-1}}    
  &\leq 
  \exp\left\{ \frac {2}{d-1}\, 2m\ln(2m) - 2m\ln(2\delta)  \right\}
  \\ &\leq
  \exp\left\{ \frac {8 m^2}{d-1}  + \frac m{\delta}   \right\}
  \ \leq g(\e) 
\end{align*}
and, with Stirling's approximation $n!>n^n e^{-n}$ and the inequality  $\frac{\ln m}{m}\le \frac{1}{e}$, we have
\begin{align*}
  n^{\frac{2}{d-1}}\left( n^{2m-d} (4 \delta)^{n} (n!) ^{-\frac 2{d-1}} \right)^{1/n}
  &\leq 
  n^{(2m-d)/n} (4 \delta e^{\frac 2{d-1}})
  \\ &\leq 
  \exp\left\{ 2m\, \frac  {\ln n}n  \right\} (4 \delta e^{\frac 2{d-1}})
  \\ &\leq  
  e^{\frac{4}{e}} 4 \Cr{sequfac} c_{\Phi} \ \epsilon^{\frac 1{2d^4}} e^{\frac 2{d-1}}
  \ = 
  f(\e)
\end{align*}
for $n >m^2$.

\subsection{Proof of Theorem \ref{thm:MainLowerBound}}

The proof of Theorem \ref{thm:MainLowerBound} is based on the following strategy: we construct a set of polytopes with $n$ facets and with bounded $\Phi$-content which we obtain by slightly perturbating a deterministic polytope which is as {\it regular} as possible.
We do so in a way which ensures that $Z$ is one of these polytopes with a high enough probability.
In Lemma \ref{lem:ConstructionBoundedPoly}, we proceed with the construction of the deterministic polytope and in Lemma \ref{lem:subsets}, we estimate the probability that $Z$ is a perturbation of this deterministic polytope.

The arguments rely on a particular assumption on the directional distribution $\varphi$.
A set $C \subset \S^{d-1}$ is called a cap of radius $r$ if it is the intersection of $\S^{d-1}$ with a ball of radius $r$ having its center on the sphere $\S^{d-1}$.
In the following we assume that $\varphi$ is well spread and thus there is a cap $C$ of radius $ r < 1 $ on the unit sphere and a constant $ \Cr{const:5} $ with 
\[
  \varphi ( \cdot )  > \Cr{const:5} \mathscr{H}^{d-1} ( \cdot ) .
\]
Without loss of generality we assume that the cap is centred at the point  $ \vect{e}_d = ( 0 , \ldots , 0 , 1 ) $.
Observe that since $\varphi$ is an even measure it is well spread on $C \cup (-C)$.

We start with two lemmata.
The first one essentially ensures that all polyhedra occurring in this section are contained in a big ball and hence are bounded.
The second lemma constructs sets $S_i \subset \HH$ such that the outer normals of the corresponding halfspaces are in $ C \cup (-C)$, their measure is of order $O(n^{-\frac{d+1}{d-1}})$, and their intersection forms a polytope with $n$ facets in $B^d$.
In the following we write $C(\vect{y}, \rho ) =B(\vect{y}, \rho) \cap \S^{d-1}$ for caps on the sphere.
\begin{lemma}
  \label{lem:ConstructionBoundedPoly}
  There exist a constant $\Cr{BoundedPoly1} = \Cr{BoundedPoly1}(d)$ and $m = m(d,r) < \Cr{BoundedPoly1}$ points $ \vect{y}_i \in C \cup (-C) $, $ i = 1 , \ldots , m$ such that the caps $C(\vect{y}_i, r/12)$ are pairwise disjoint and
  \[
    \bigcap\limits_{i=1}^m  H ( \vect{v}_i , 1 )^-
    \subset B ( \origin , 4 r^{-1} )
  \]
  for any $ \vect{v}_i \in C(\vect{y}_i, r/12)  \cap (C \cup (-C)) $, $ i = 1 , \ldots , m$.
\end{lemma}
\begin{proof}
  We choose a saturated packing of caps $C(\vect{y}_i, r/12) $ with $\vect{y_i}\in C \cup(-C)$, $ i = 1,\ldots,m$.
  Here we call a packing saturated if there is no possibility for adding another ball of radius $r/12$.
  Since the curvature of the sphere becomes negligible when $r\to0$, we have that $m$ is of the same order as a saturated packing of $(d-1)$-dimensional balls of radius $r/12$ in $r B^{d-1}$.
  Clearly this is independent from $r$ and therefore $m< \Cr{BoundedPoly1}$ for some constant $\Cr{BoundedPoly1}$ depending only on $d$.
  
  This implies first that $\bigcup C(\vect{y}_i, r/6)$ is a covering of $C \cup (-C)$. 
  Second, each cap $C(\vect z, r/4)$, $\vect z \in C$ contains one of the caps $C(\vect{y}_i, r/12)$, because $\vect z \in C(\vect y_i, r/6)$ for some $i =1, \dots, m$. 
  
  The rest of the proof follows from explicit geometric calculations.
  Assume in the contrary that there are $\vect v_i \in C(\vect y_i, r/12) \cap (C \cup (-C))$ such that
  $$    \bigcap\limits_{i=1}^{m} H ( \vect{v}_i , 1 )^- \nsubseteq B ( \origin , 4 r^{-1} ). $$
  This in particular implies that either 
  \[
    \vect{e}_d^\perp \cap \bigcap\limits_{\vect{v}_i \in C} H ( \vect{v}_i , 1 )^- \nsubseteq B ( \origin , 4 r^{-1} )
    \ \text{ or } \ 
    \vect{e}_d^\perp  \cap \bigcap\limits_{\vect{v}_i \in -C} H ( \vect{v}_i , 1 )^- \nsubseteq B ( \origin , 4 r^{-1} ).
  \]
  Recall that $C$ is a cap with center $\vect{e}_d$.
  Without loss of generality assume that\\$\vect x = (4 r^{-1}, 0, \dots, 0)$ is a point with $\| \vect x\| = 4 r^{-1}$ which is contained in $\bigcap_{\vect{v}_i \in C}  H ( \vect{v}_i , 1 )^-$.
  Let us define $\vect x_0= (r/4, 0, \dots, 0, \sqrt{1-r^2/16})$.
  By elementary trigonometric calculations the line through $\vect x$ and $\vect x_0$ is tangent to the sphere at $\vect x_0$.
  Because $\vect x$ is contained in $\bigcap H ( \vect{v}_i , 1 )^-$, none of the points $\vect v_i$ may be contained the cap $C_{\vect x}= C(\vect e_1, \| \vect e_1 - \vect x_0\|) $. 
  
  Next observe that the point $\vect x_C=(\sqrt{1-h^2}, 0, \dots, 0, h)$ with $h=1 - r^2/2$ is on the relative boundary of $C$ and in $C_{\vect x}$, and
  \[
    \|\vect x_C - \vect x_0\|
    \geq 
    \sqrt{1-h^2}  - \frac 14 r 
    \geq 
    \frac 34 r-  \frac 14 r 
    \geq \frac 12 r. 
  \]
  Hence $C \cap C_{\vect x}$ contains a cap of radius $r/4$.
  Yet this cap must contain one of the caps $C(\vect y_i, r/12)$ and thus one of the points $\vect v_i$, a contradiction. 
\end{proof}

In the following lemma we assume that there exists a cap $ C $ of radius $ r \in ( 0 , 1 ) $ of the sphere and a constant $ \Cr{const:5} $ with $ \varphi ( \cdot ) > \Cr{const:5} \mathscr{H}^{d-1} ( \cdot ) $ on $ C $.

\begin{lemma}
  \label{lem:subsets}
  There exists a constant $\Cl[univ]{vol} $ such that the following holds.
  For any $ n > \Cr{BoundedPoly1} $, there are disjoint subsets $ S_1 , \ldots , S_n \subset \HH  $ with 
  \[
    \mu (S_i) > \Cr{const:5} \Cr{vol} r^{d+2} n^{- \frac{d+1}{d-1}} 
  \]
  and for $ H_1 \in S_1 , \ldots , H_n \in S_n $ we have 
  \[
    \bigcap_i H_i^- \in \PPn
  \]
  and
  \begin{equation}
    \label{condition:3}
    \bigcap_i H_i^- \subset \B .
  \end{equation}
\end{lemma}
\begin{proof}
  Consider the $ m < \Cr{BoundedPoly1} $ caps $C(\vect y_i, r/12)$ which have been constructed in Lemma \ref{lem:ConstructionBoundedPoly}, and fix $ n > \Cr{BoundedPoly1}$.
  In each of the sets $C(\vect y_i, r/12) \cap (C\cup-C)$ we produce an optimal packing of $\lceil n/m \rceil$ smaller caps $C(\vect z_j,\rho)$ where we can choose $ \rho$ such that it satisfies 
  \[
    \Cl[univ]{radiusrho} n^{- \frac 1{d-1}} r
    \le \rho \le \frac{r}{12} 
  \]
  with a constant $\Cr{radiusrho}$ independent of $r, n$ and $m$. Observe that the number of caps constructed in this way is between 
  $n $ and $n+m $.
  We choose precisely $n$ of these caps $C(\vect z_i, \rho)$ in such a way that in each set $C(\vect y_i, r/12)\cap (C\cup-C)$ there is at least one cap $C(\vect z_i, \rho)$.
  
  As already used above, a cap of radius $t$ has height $t^2/2$.
  Let $\vect v_i$ be arbitrary points in $C(\vect z_i, \rho/2)$, $i=1, \dots ,n$.
  Since each cap 
  \[
    C\left(\vect v_i, \frac {\rho}2  \right)
    = H \left(\vect v_i, 1-\frac 12 \Big( \frac{\rho}2 \Big)^2    \right)^+ \cap \S^{d-1}
  \]
  is contained in the cap $ C(\vect z_i, \rho)$, it is disjoint from all other caps $ C(\vect z_j, \rho)$, and thus also disjoint from all other caps $ C(\vect v_j, \rho/2)$.
  Hence for arbitrary $r_i $ with $ 0 \leq r_i \leq \rho/2$, all points $ (1- r_i^2/2) \, \vect v_i$ are on the boundary of $\cap_{i=1}^n  H ( \vect{v}_i , 1- r_i^2/2)^-$ and thus this intersection has $n$ facets.
  
  Since each set $C(\vect{y}_i,r/12)$ contains a cap $C(\vect{z}_i,\rho)$, there are $m$ points $\vect{v}_i$,\linebreak[4] {$\vect{v}_1, \cdots, \vect{v}_m$} say, which belong to $C(\vect{y}_1,r/12),\cdots,C(\vect{y}_m,r/12)$ respectively.
  Combining Lemma \ref{lem:ConstructionBoundedPoly} applied to $\vect{v}_1,\cdots,\vect{v}_m$ and the considerations above, we obtain:
  there are pairwise disjoint sets 
  \[
    T_i=
    \left\{ H(\vect v, t):\ \vect v \in C \Big(\vect{z}_i, \frac \rho{2}\Big),\ t \in \Big[1-\frac 12 \, \Big(\frac \rho {2} \Big)^2  ,1 \Big]\right\} \subset \HH , \ i=1, \dots, n,
  \]
  such that for an arbitrary $n$-tuple $H(\vect v_i, t_i) \in T_i$, $i=1, \dots, n$, we have
  \[
    \bigcap\limits_{i=1}^n  H ( \vect{v}_i , t_i)^-
    \subset \bigcap\limits_{i=1}^m  H ( \vect{v}_i , t_i)^- \subset B ( \origin , 4 r^{-1} ) 
    \ \ \text{ and }\  \ 
    \bigcap\limits_{i=1}^n  H ( \vect{v}_i , t_i)^- \in \PP_n .
  \]
   
  We normalize such that $B ( \origin , 4 r^{-1} )  $ is replaced by the unit ball and define
  \[ 
    S_i = 
    \left\{ H(\vect v, t):\ \vect v \in C \Big(\vect{z}_i, \frac \rho{2}\Big), t \in \frac r4\Big[1-\frac 12 \, \Big(\frac \rho {2} \Big)^2  ,1 \Big]\right\} 
    = \frac r4 T_i 
    \subset \HH 
  \]
  for $i=1, \dots, n.$.
  The sets $S_i$ have measure at least 
  \[
    \mu(S_i) \geq \Cr{const:5} \mathscr{H}^{d-1}\Big(C \Big(\vect{z}_i, \frac \rho{2}\Big)  \Big) \frac{ r \rho^2 }{ 32 }
    \geq   
    \Cr{const:5} \Cr{vol} r^{d+2} n^{- \frac{d+1}{d-1}} 
  \]
  since $\rho \geq \Cr{radiusrho} n^{- \frac 1{d-1}} r  $, where $\Cr{vol}$ is a constant depending only on $d$.
  This yields $n$ sets $S_i$ with the desired properties.
\end{proof}

We point out that because of Lemma \ref{lem:subsets}  and condition \eqref{condition:3} therein, having $H_{\sigma(1)}\in S_1,\cdots,H_{\sigma(n)}\in S_n$ for some permutation $\sigma\in\mathfrak{S}_n$ implies that $\Phi(\cap_{i=1}^nH_i^{-})<1$, $\cent (\cap_{i=1}^nH_i^{-})\in \B$ and $\cap_{i=1}^nH_i^{-}\in \PPn$.
Using this and \eqref{eq:ProbaIntegralFormbis} yields
\begin{align*}
  \P ( f(Z)=n) 
  &=
    \frac{ \gamma^d }{ \IntMos \kappa_d} \frac{(n-d)!} {n!}
    \int\limits_{\HS^n}
    \1 \left( \cent (\cap_{i=1}^nH_i^{\epsilon_i}) \in \B
    \right)
    \1 \left( \Phi(\cap_{i=1}^nH_i^{\epsilon_i}) <1  \right)
    \\& \phantom{{} = \frac{ \gamma^d }{ \IntMos \kappa_d} \frac{(n-d)!} {n!} \int\limits_{\HS^n} {}}
    \1 \left( \cap_{i=1}^nH_i^{\epsilon_i} \in \PPn \right) 
    \MeasMultiSimple{\MeasHalfSpaceMu}{n}{H^\epsilon},
  \\&\geq
    \frac{ \gamma^d }{ \IntMos \kappa_d} (n-d)!
    \int\limits_{\HH ^n}
    \1 \left( H_1 \in S_1 \right)
    \cdots
    \1 \left( H_n \in S_n \right) 
    \MeasMultiSimple{\mu}{n}{H}
  \\&=
    \frac{ \gamma^d }{ \IntMos \kappa_d}  (n-d)!  \ 
    \mu ( S_1 ) \cdots \mu ( S_n ) 
  \\&> 
    \frac{ \gamma^d }{ \IntMos \kappa_d}  (n-d)!  \ \left( \Cr{const:5} \Cr{vol} r^{d+2} n^{- 1- \frac{2}{d-1}}  \right)^n 
\end{align*}
for $n \geq \Cr{BoundedPoly1}$.
Stirling's approximation  $ n! > n^{n} e^{-n} $ implies 
\[
  \frac{ ( n - d )! }{ n^n }  
  > \frac{n!}{n^{n+d}} 
  > e^{-(d+1)n}  
\]
With $\Cr{LowerBound1} = \min\left( 1 , \frac1{\kappa_d} \right) \Cr{vol} e^{-(d+1)} $, this implies immediately the statement of Theorem \ref{thm:MainLowerBound}.

\section{Big Cells}\label{sec:bigcells}
In this section we are interested in the behaviour of the typical cell $Z$ when $\Phi (Z)$ tends to infinity.
In particular we aim at proving results on the asymptotic behaviour of $\P(\Phi(Z) > a)$ (Theorem \ref{mainthm:DistribPhi}), on the shape of such big cells  in the general case (Theorem \ref{mainthm:FlatBigPhi}) and on the existence of a limit shape in the particular case when $\varphi$ is concentrated on a finite set of points (Theorem \ref{mainthm:nmaxshape}).

To get Theorem \ref{mainthm:FlatBigPhi}, we need a new upper-bound for the probability of the event $\{\Phi(Z)>a\}$ intersected with the event that  the cell is $(\e \colon i,j)$-elongated, which is given below.
\begin{theorem}
  \label{thm:boundPhiFlat}
  Assume  $1 \leq i < j \leq \lceil (d-1) / 2 \rceil $.  
  There exist constants $\Cl[univ]{eps<}$ and $\Cl[univ]{lemmaelong2}$, such that for any $ \epsilon< \Cr{eps<} $ and for any $a \geq \gamma ^{-1} \e^{- (2d+3)}, $
  \[ 
    \P \left( \Phi(Z) > a ,\ \frac{V_j(Z)^{\frac 1j}}{V_i(Z)^{\frac 1i}} < \epsilon \right)
    \leq
    \exp \left( -\gamma a+  \Cr{lemmaelong2} \epsilon^{\frac {1}{6d^4}} (\gamma a)^{\frac {d-1}{d+1}}   \right)  .
  \]
\end{theorem}
Actually, the bound in Theorem \ref{thm:boundPhiFlat} is close to the upper-bound from Theorem \ref{mainthm:DistribPhi} but with a constant in front of $(\gamma a)^{\frac {d-1}{d+1}}$ which is arbitrarily small.

In the next subsection, we show how to deduce easily Theorem \ref{mainthm:FlatBigPhi} from Theorems \ref{mainthm:DistribPhi} and \ref{thm:boundPhiFlat}.
The rest of Section \ref{sec:bigcells} is devoted to the proof of Theorems \ref{mainthm:DistribPhi}, \ref{thm:boundPhiFlat} and \ref{mainthm:nmaxshape}.

\subsection{Deducing Theorem \ref{mainthm:FlatBigPhi} from Theorems \ref{mainthm:DistribPhi} and \ref{thm:boundPhiFlat}}
Let us assume that $\varphi$ is well spread and that  $1 \leq i < j \leq \lceil (d-1)/2 \rceil$.
Then the lower-bound from Theorem \ref{mainthm:DistribPhi} together with Theorem \ref{thm:boundPhiFlat} imply that for any $ \epsilon< \Cr{eps<} $ and for any $a \geq \gamma ^{-1} \max\{ \Cr{const:LowerA}, \e^{- (2d+3)}\}, $ we have
\[ 
  \P \left( \ \frac{V_j(Z)^{\frac 1j}}{V_i(Z)^{\frac 1i}} < \epsilon \mid \Phi(Z) > a \right)
  \leq
  \exp \left( \Big( \Cr{lemmaelong2} \epsilon^{\frac {1}{6d^4}} - \Cr{const:LowerBoundPhiContent}  \Big)(\gamma a)^{\frac {d-1}{d+1}}  \right) .
\]
For $\epsilon<\left(\frac{\Cr{const:LowerBoundPhiContent}}{\Cr{lemmaelong2}}\right)^{6d^4}$, the conditional probability above goes to zero when $a$ goes to $\infty$. This shows Theorem \ref{mainthm:FlatBigPhi}.

\subsection{Proof of Theorem \ref{mainthm:DistribPhi}}
\label{sec:ProofDistribPhi}
We start with three intermediary lemmas: Lemma \ref{lem:SumRl} builds upon the Complementary Theorem to get a rewriting of the distribution tail of $\Phi(Z)$ as a function of the distribution tail of $f(Z)$.
In Lemma \ref{lem:ApproxRbyQ}, we deduce from Theorems \ref{mainthm:UpperBoundfn} and \ref{mainthm:LowerBoundfn} respectively upper and lower-bounds for the distribution tail of $f(Z)$.
Finally, Lemma \ref{lem:upperboundExpModif} contains analytical estimates for some subexponential power series.  

In the sequel, we use the abbreviations $ q_n := \P ( f(Z)=n ) $ and $ r_n := \sum_{ k \geq n} q_k $ for every $n\ge (d+1)$.

In the following lemma, we rewrite the probability $\P (\Phi(Z) > a )$ as a power series in $a$.
\begin{lemma}\label{lem:SumRl}
  For every $a>0$, we have
  \[
    \P (\Phi(Z) > a )
    = e^{-\gamma a} \sum _{ n \geq 0 } r_{n+d+1} \frac { (\gamma a)^{n} } { n! }  
  \]
\end{lemma}
\begin{proof}
  Because of the Complementary Theorem \ref{thm:ComplementaryThm} we have for every $a>0$
  \begin{align*}
    \P (\Phi(Z) > a )
    &= \sum _{ n \geq d+1 } q_{n} \P ( \Phi(Z) > a \mid  f(Z)=n  )
    \\
    &= \sum _{ n \geq d+1} q_{n} \int\limits_a^\infty e^{-\gamma t} \frac { \gamma^{n-d} t^{n-d-1} } { (n-d-1)! } \dint t.
  \end{align*}
  Now we recall that iterated integrations by parts show that for every $n\ge (d+1)$,
  \[
    \int\limits_a^\infty e^{-\gamma t} \frac{\gamma^{n-d} t^{n-d-1}}{(n-d-1)!} \dint t
    = e^{-\gamma a} \sum_{m=0}^{n-d-1} \frac{(\gamma a)^m}{m!} .
  \]
  Consequently, we obtain that
  \[
    \P (\Phi(Z) > a )
    = e^{-\gamma a} \sum _{ n \geq d+1 } \sum _{m=0} ^{n-d-1} q_{n} \frac { (\gamma a)^{m} } { m! }
    = e^{-\gamma a} \sum _{ m \geq 0 } r_{m+d+1} \frac { (\gamma a)^{m} } { m! },
  \]
  which shows Lemma \ref{lem:SumRl}.  
\end{proof}

The relation from Lemma \ref{lem:SumRl} indicates that in order to bound $\P(\Phi(Z)>a)$, we need to find bounds for $r_{n+d+1}$. This is done in the next lemma.
\begin{lemma}
  \label{lem:ApproxRbyQ}
  There exists a constant $\Cl[univ]{rqmult}$ depending on $\varphi$ such that for $n\geq 0$ we have
  \[ 
    r_{n+d+1} 
    <  \Cr{rqmult}^{n} (n!)^{-\frac 2{d-1}}.
  \]
  Assume that $\varphi$ is well spread.  Then there exists a constant $\Cl[univ]{const:LowerBound} >0 $ depending on $\varphi$ such that for $n\geq 0$ we have
  \[ 
    r_{n+d+1} 
    \geq \Cr{const:LowerBound}^n (n!)^{-\frac 2{d-1}} .
  \]
\end{lemma}
\begin{proof}
  We start with the upper-bound.
  By Theorem \ref{mainthm:UpperBoundfn} we have for $n\geq d+1$,
  \begin{equation}
    \label{eq:UpperBoundBis}
    q_n <  \Cr{upperbound2TH1}^n \, n^{-\frac{2n}{d-1}}.
  \end{equation}
  with some constant $\Cr{upperbound2TH1}>0 $ depending on $\varphi$. 
  By \eqref{eq:UpperBoundBis} we have, 
  \begin{align*}
    r_{n+d+1} 
    & \leq 
    \sum_{k \geq n+d+1} \Cr{upperbound2TH1}^k \, k^{-\frac{2}{d-1}k} 
    \\&\leq
    \Cr{upperbound2TH1}^{n} n^{-\frac{2}{d-1}n}  \sum_{k \geq d+1} \Cr{upperbound2TH1}^k \, k^{-\frac{2}{d-1}k} .
  \end{align*}
  We use $n^{-n} \leq (n!)^{-1}$, and observe that the remaining sum is convergent and independent of $n$. Hence in order to get the upper-bound, it suffices to set 
  \[
    \Cr{rqmult}
    := \Cr{upperbound2TH1} \max\left\{1, \sum_{k \geq d+1} \Cr{upperbound2TH1}^{k} \, k^{-\frac{2}{d-1} k} \right\}  . 
  \]
  We assume now that $\varphi$ is well spread and prove the lower-bound for $r_{n+d+1}$.
  Theorem \ref{mainthm:LowerBoundfn} tells us that when $\varphi$ is well spread, for every $n\ge 0$,
  \[
    q_{n+d+1}
    > \Cr{lowerbound2TH2}^{n+d+1} \,  (n+d+1)^{- \frac{2(n+d+1)}{d-1}}
  \]
  Consequently, using Stirling's approximation $n^{-n}> e^{-n}(n!)^{-1}$ and the simple inequality $r_{n+d+1}>q_{n+d+1}$, we get 
  \begin{align*}
    r_{n+d+1}
    &> \left(\Cr{lowerbound2TH2} e^{-\frac{2}{d-1}}\right)^{n+d+1}[(n+d+1)!]^{-\frac{2}{d-1}}
    \\
    &> \left(\Cr{lowerbound2TH2} e^{-\frac{2}{d-1}}\right)^{n+d+1}[(n+d+1)^{d+1}\cdot n!]^{-\frac{2}{d-1}}
    \\  
    &> \left(\Cr{lowerbound2TH2} (d+1)^{-\frac{2}{d-1}}e^{-\frac{2}{d-1}}\right)^{n+d+1}(n!)^{-\frac{2}{d-1}}
  \end{align*}
  because $(n+d+1)^{d+1} <  (d+1)^{n+d+1}$ for $n+d+1 \geq d+1 \geq 3$.
  
  Taking $\Cr{const:LowerBound}=\Cr{lowerbound2TH2} (d+1)^{-\frac{2}{d-1}}e^{-\frac{2}{d-1}}\min(1,(\Cr{lowerbound2TH2} (d+1)^{-\frac{2}{d-1}}e^{-\frac{2}{d-1}})^{d+1})$, we get the required result. 
\end{proof}

The combination of the two previous lemmas implies that $\P(\Phi(Z)>a)$ is well approximated by subexponential power series of type $\sum_{n\geq 0} \frac{x^n}{(n!)^{\alpha}}$.
The next lemma, which is purely analytical, investigates the behaviour of such power series. 
\begin{lemma}
  \label{lem:upperboundExpModif}
  For any $\alpha>1$, we have  
  \[
    \exp \left( \frac12 \alpha x^{\frac 1\alpha} \right)
    < \sum_{n\geq d+1} \frac{x^n}{(n!)^{\alpha}} 
    < \sum_{n\geq 0} \frac{x^n}{(n!)^{\alpha}} 
    < \exp \left( \alpha x^{\frac 1\alpha} \right)  
  \] 
  where the first inequality holds for $x \geq (2(3d+5))^{\alpha} $ and the second for all $x>0$. 
\end{lemma}
\begin{proof}
  The right hand side inequality follows immediately from the following simple computations.
  \[
    \sum_{n\geq d+1} \frac{x^n}{(n!)^{\alpha}} 
    <  \sum_{n\geq0} \left( \frac{(x^{\frac 1\alpha})^n}{n!} \right)^{\alpha}
    < \left( \sum_{n\geq0} \frac{(x^{\frac 1\alpha})^n}{n!} \right)^{\alpha}
    = \exp \left( \alpha x^{\frac 1\alpha} \right).
  \]
 For the left hand side inequality, H{\"o}lder's inequality gives for any finite $I \subset \N \setminus [d+1] $
  \begin{equation}
    \label{eq:Holder}
    \sum_{n\geq d+1} \left( \frac{(x^{\frac 1\alpha})^n}{n!} \right)^{\alpha} \geq
    \sum_{n\in I} \left( \frac{(x^{\frac 1\alpha})^n}{n!} \right)^{\alpha} 
    \geq |I|^{-(\alpha-1)} \left( \sum_{n \in I} \frac{(x^{\frac 1\alpha})^n}{n!} \right)^\alpha .
  \end{equation}

  For $Y$ a Poisson distributed random variable with mean $\lambda$ it is well known, and can be proved e.g. by Chebishev's inequality, that for $ I   = ( \lambda - \sqrt {2\lambda}, \lambda + \sqrt {2\lambda}) \cap \N$, we have
  \[
  \sum_{ n \in I } e^{-\lambda} \frac{\lambda^n}{n!} 
  = 1 - \P \left( |Y-\lambda| \geq \sqrt{2 \lambda } \right)
    \geq  \frac 12.
  \]
  $ I $ has at most $ 2 \sqrt{ 2 \lambda} + 1 < 4 \sqrt{\lambda} $ elements, when $\lambda \geq 1$. Putting this for 
  $\lambda = x^{1/\alpha} $ into \eqref{eq:Holder} yields
  \[
    \sum_{n\geq d+1} \left( \frac{(x^{\frac 1\alpha})^n}{n!} \right)^{\alpha} 
      \geq
    \left(4 x^{\frac 1{2\alpha}}\right)^{-(\alpha-1)} \left( e^{x^{\frac 1 \alpha}}\frac 12\right)^{\alpha} 
    \geq
    \left(8^{-\alpha}x^{-\frac 12} \right) e^{\alpha x^{ \frac 1 \alpha}} 
  \]
  as long as the condition $x^{1/\alpha} - \sqrt {2 x^{1/\alpha}} \geq d+1$ is fulfilled.
  Observe that $ x \geq (3d+5)^\alpha $ implies $x^{1/(2\alpha)} \geq \sqrt {d+2} +1$ which in turn implies $x^{1/\alpha} - 2 x^{1/(2\alpha)} +1 \geq d+2$ which gives the required condition.
  
  For $ t \geq 3$ we have $ 2 \ln 8 + t \leq 1 + t + t^2/2 \leq e^t $, or equivalently
  \[
    - \alpha \ln 8 - \frac 12 \ln x    
    \geq
    - \frac 12 \alpha x^{\frac 1\alpha} ,\ 
    \textrm{ i.e. }\ 
    8^{- \alpha}   x^{- \frac 12}    
    \geq
    e^{- \frac 12 \alpha x^{1/\alpha}} 
  \]
  for $ x^{1/{\alpha } } \geq e^3$.
  The inequality $2(3d+5) > e^3$ concludes the proof.
\end{proof}

\medskip
We are now ready to prove Theorem \ref{mainthm:DistribPhi}. Combining Lemma \ref{lem:SumRl} and the upper-bound of Lemma \ref{lem:ApproxRbyQ}, we get
\begin{equation*}
  \P ( \Phi(Z) > a ) 
  < e^{-\gamma a}  \sum_{n \geq 0} \Cr{rqmult}^{n} \frac{(\gamma a)^n}{(n!)^{\frac{d+1}{d-1}}} .
\end{equation*}
Applying now Lemma \ref{lem:upperboundExpModif} to $x=\Cr{rqmult} \gamma a$ and $\alpha=\frac{d+1}{d-1}$, we obtain that
\[
  \P ( \Phi(Z) > a )
  < e^{-\gamma a} \sum_{n \geq 0} \frac{(\Cr{rqmult} \gamma a)^{n} }{ (n!)^{\frac{d+1}{d-1}} }
  < \exp \left( -\gamma a  + \frac{d+1}{d-1}  (\Cr{rqmult}  \gamma a)^{\frac{d-1}{d+1}} \right).
\]
The proof of the lower-bound is nearly identical and we leave the details to the reader.

\subsection{Proof of Theorem \ref{thm:boundPhiFlat}}
Assume $1 \leq i < j \leq \lceil (d-1) / 2 \rceil $.
In the sequel, we use the notation $q_n^\e := \P \left( f(Z)=n, \ \frac{V_j(Z)^{\frac 1j}}{V_i(Z)^{\frac 1i}}< \epsilon\right)$ and $r_n^\e:=\sum_{k\ge n}q_k^\e$,  for every $n\ge (d+1)$ and $\e>0$.
The proof follows along the same lines as the upper bound of Theorem \ref{mainthm:DistribPhi} with minor adaptations.
Indeed, we need some analogues to the statements of Lemmas \ref{lem:SumRl} and \ref{lem:ApproxRbyQ} when $q_n$ is replaced by $q_n^\e$, i.e. when the extra-condition that $Z$  is $(\e \colon i,j)$-elongated is added. 

The lemma below is a rewriting of the joint distribution of $(\shape(Z),\Phi(Z))$ as a power series.
\begin{lemma}
  \label{eq:SumRlShape}
  For any measurable set of shapes $S\subset \KKcs$ and $a>0$, we have
  \begin{align*}
    \P (\shape(Z) \in S, \, \Phi(Z) > a )
    &= e^{-\gamma a} \sum_{k\geq d+1} \P (\shape(Z) \in S, \, f(Z) = k ) \sum _{l=0} ^{k-d-1} \frac { (\gamma a)^{l} } { l! }
    \\
    &= e^{-\gamma a} \sum_{l\geq0} \P (\shape(Z) \in S, \, f(Z) \geq l+d+1 ) \frac{(\gamma a)^l}{l!}.
  \end{align*}
\end{lemma}
The proof of this result is fully analogous to that of Lemma \ref{lem:SumRl} and is therefore omitted.

As in Lemma \ref{lem:ApproxRbyQ}, we require now an upper-bound for $r_{n+d+1}^\e$.
\begin{lemma}
  \label{lemma:boundsManyFacetsIsoperimeter}
  Assume $1 \leq i < j \leq \lceil (d-1) / 2 \rceil $.
  There exist constants $\Cr{eps<}$ and $\Cl[univ]{lemmaelong}$ depending on $\varphi$, such that for any $ \epsilon  < \Cr{eps<}  $ we have 
  \[
    r_{n+d+1}^\e  <   e^{ \Cr{thmelong1} \epsilon^{-2(d-1)} } 
    (\Cr{lemmaelong} \epsilon^{\frac 1{2d^4}})^{n} (n!)^{-\frac 2{d-1}} 
  \]
for $n \geq 0$.
\end{lemma}
\begin{proof}
  Theorem \ref{thm:boundsManyFacetsIsoperimeter} implies that for any  $ \epsilon  < \Cr{elong1}^{2/(d-1)}\Cr{elong2}^{-1} (c_{\Phi})^{-1}   $ and $n > \lfloor \Cr{elong1} \epsilon^{-(d-2)} \rfloor^2 $, 
  \[
    q_n^\e <   \frac{ \gamma^d }{ \IntMos } 
    \ e^{ \Cr{thmelong1} \epsilon^{-2(d-1)}} (\Cr{thmelong2}\epsilon^{\frac 1{2d^4}})^n \  n^{-\frac{2n}{d-1}}  .
  \]
  For $\epsilon <\Cr{eps<}$ we have 
  \begin{align*}
    r_{n+d+1}^\e 
    &\leq
    \frac{ \gamma^d }{ \IntMos }   \ e^{ \Cr{thmelong1} \epsilon^{-2(d-1)}} 
    \sum_{k \geq n+d+1} \left(\Cr{thmelong2}\epsilon^{\frac 1{2d^4}}\right)^k \  k^{-\frac 2{d-1} k}  
    \\ &\leq
    \frac{ \gamma^d }{ \IntMos }   \ e^{ \Cr{thmelong1} \epsilon^{-2(d-1)}} 
    (\Cr{thmelong2}\epsilon^{\frac 1{2d^4}})^{n} \  n^{-\frac{2n}{d-1}}  
    \sum_{k \geq d+1} \left(\Cr{thmelong2}\Cr{eps<}^{\frac 1{2d^4}} \right)^{k} \  k^{-\frac 2{d-1} k}  .
  \end{align*}
  We use $n^{-n} \leq (n!)^{-1}$, and observe that the remaining sum is convergent and independent of $n$.
  Hence it suffices to set 
  \[
    \Cr{lemmaelong}
    := \Cr{thmelong2} \max\left\{1, \sum_{k \geq d+1} \left(\Cr{thmelong2}\Cr{eps<}^{\frac 1{2d^4}} \right)^{k} \, k^{-\frac{2}{d-1} k} \right\}  . 
  \]
\end{proof}

Let us now proceed with the proof of Theorem \ref{thm:boundPhiFlat}.
Applying Lemma \ref{eq:SumRlShape} to the set $ S =\{ K \in \KKcs \colon V_j(Z)^{1/j} V_i(Z)^{-1/i} < \epsilon \} $, we get
\begin{align*}
  \P \left( \Phi(Z) > a ,\ \frac{V_j(Z)^{\frac 1j}}{V_i(Z)^{\frac 1i}} < \epsilon \right)
  &= e^{-\gamma a} \sum _{ n \geq 0 } r^\e_{n+d+1} \frac { (\gamma a)^{n} } { n! }. 
\end{align*}
We combine this with Lemma \ref{lemma:boundsManyFacetsIsoperimeter} to deduce
\begin{align*}
  \P \left( \Phi(Z) > a ,\ \frac{V_j(Z)^{\frac 1j}}{V_i(Z)^{\frac 1i}} < \epsilon \right)
  & \leq 
  e^{-\gamma a+ \Cr{thmelong1} \epsilon^{-2(d-1)} } 
  \sum _{ n \geq 0 } (\Cr{lemmaelong} \epsilon^{\frac 1{2d^4}})^{n} \frac { (\gamma a)^{n} } { (n!)^{\frac {d+1}{d-1}} } \, .  
\end{align*}
Lemma \ref{lem:upperboundExpModif} ends the proof:
\begin{align*}
  \P \left( \Phi(Z) > a ,\ \frac{V_j(Z)^{\frac 1j}}{V_i(Z)^{\frac 1i}} < \epsilon \right)
  & \leq
  \exp \left( -\gamma a+  \frac {d+1}{d-1} (\Cr{lemmaelong} \epsilon^{\frac 1{2d^4}} \gamma a)^{\frac {d-1}{d+1}}  + \Cr{thmelong1} \epsilon^{-2(d-1)} \right)  
  \\  & \leq
  \exp \left( -\gamma a+ \Cr{lemmaelong2} \epsilon^{\frac {1}{6d^4}} (\gamma a)^{\frac {d-1}{d+1}}  \right)  
\end{align*}
for $\e^{- (2d+3)}\leq \gamma a$ because this implies $ \e^{-2(d+1)} \leq \e \gamma a \leq \e^{\frac{1}{2d^4}} \gamma a $ and thus $ \e^{-2(d-1)} \leq (\e^{\frac{1}{2d^4}} \gamma a)^{\frac{d-1}{d+1}} \leq \e^{\frac{1}{6d^4}} ( \gamma a)^{\frac{d-1}{d+1}} $ since $\frac{d-1}{d+1}\geq \frac13$.

\subsection{Proof of Theorem \ref{mainthm:nmaxshape}}

Assume $\varphi$ is concentrated on a finite number $n_{\max}$ of points.
Thus $f(Z) \leq n_{\max}$ with probability one. 
We use again the notation $q_n = \P( f(Z)=n)$, and fix some subset $S \subset \KKcs$ of the shape space such that $\P (\shape(Z) \in S, \, f(Z) = n_{\max} )  >0$.
Because of Lemma \ref{eq:SumRlShape}, we have
\begin{align*}
  \P (\shape(Z) \in S, \, \Phi(Z) > a )
  &= 
  e^{-\gamma a} \sum _{ k \leq n_{\max} } \P (\shape(Z) \in S, \, f(Z) = k ) \sum _{l=0} ^{k-d-1} \frac { (\gamma a)^{l} } { l! }
  \\ &=
  e^{-\gamma a} \P (\shape(Z) \in S, \, f(Z) = n_{\max} )  \sum _{l=0} ^{n_{\max}-d-1} \frac { (\gamma a)^{l} } { l! } \  (1+ O ((\gamma a)^{-1}))
  \\ &=
  \P(\shape(Z) \in S, \, \Phi(Z)\geq a, \ f(Z)=n_{\max})(1+O((\gamma a)^{-1}))
\end{align*}
This implies
\begin{align*}
  \P(\shape(Z) \in S | \Phi(Z)\geq a) 
  &=
  \frac{ \P(\shape(Z) \in S,\  \Phi(Z)\geq a)}{\P(\Phi(Z)\geq a) } 
  \\&=
  \frac{ \P(\shape(Z) \in S,\  \Phi(Z)\geq a, \ f(Z)=n_{\max}) (1+O(\gamma a^{-1}))}{\P(\Phi(Z)\geq a , \ f(Z)=n_{\max}) (1+O((\gamma a)^{-1}))} 
  \\&=
  \frac{ \P(\shape(Z) \in S,\  \Phi(Z)\geq a \mid \ f(Z)=n_{\max}) }{\P(\Phi(Z)\geq a \mid \ f(Z)=n_{\max}) } (1+O((\gamma a)^{-1}))
  \\&=
  \P(\shape(Z) \in S | f(Z)=n_{\max})(1+ O((\gamma a)^{-1}))
\end{align*}
where in the last equation we used again the Complementary Theorem \ref{thm:ComplementaryThm}.

\end{document}